\documentclass[authoryear,hyperref,final,1p]{elsarticle}
\makeatletter
\def\ps@pprintTitle{%
     \let\@oddhead\@empty
     \let\@evenhead\@empty
     \def\@oddfoot{}%
     \let\@evenfoot\@oddfoot}
\makeatother
\usepackage{graphicx}

\usepackage{amsmath}
\usepackage{subfigure}
\usepackage{amssymb}
\usepackage{mathptmx}      % use Times fonts if available on your TeX system

\RequirePackage{natbib}
\RequirePackage{hypernat}
\usepackage{url}
\RequirePackage[hyperindex,bookmarks,bookmarksnumbered,plainpages=false]{hyperref}
\RequirePackage{ifthen}

\newlength{\ei}\ei=0.0138888889em   % 1/72 em  % ca.1/20 mm
\newlength{\SyW}   \newlength{\msu}  \msu=\mathsurround % \AzP=\msu gesetzt

 \newcommand{\Ts}{\textstyle}
\newcommand{\Ss}{\scriptstyle}  \newcommand{\SSs}{\scriptscriptstyle}
\newcommand{\rfia}[1]{\makebox[\parindent][l]{%
                     \makebox[0em][r]{\rm(}\sf#1\rm)}}
\newcounter{ABCcB}
\newcommand{\theABCcC}{\alph{ABCcB}}

\newcommand{\maxev}{\mathop{\rm max\hspace*{6\ei}ev}\nolimits}

\newcommand{\Ew}{\mathop{\rm {{}E{}}}\nolimits} % mathop erkennt nur (
\newcommand{\ve}{\varepsilon}

\newcommand{\ssr}{\rm\scriptscriptstyle}

\newcommand{\Lo}{\mathop{\rm {{}o{}}}\nolimits}
\newcommand{\LO}{\mathop{\rm {{}O{}}}\nolimits}

\newcommand{\tr}{\mathop{\rm{} tr{}}}
\newcommand{\B}{\mathbb B}

\newcommand{\R}{\mathbb R}
\newcommand{\Rq}{\bar{\mathbb R}}
\newcommand{\Bq}{\bar{\mathbb B}}

\newcommand{\N}{\mathbb N}
\newcommand{\Jc}{\mathop{\bf\rm{{}I{}}}\nolimits}
\newcommand{\EM} {{\mathbb I}}

\newcommand{\sR}{{\Ss \mathbb R}}
\newcommand{\sRq}{\bar{\Ss \mathbb R}}

\newcommand{\sEM} {{\Ss \mathbb I}}
\newcommand{\also}{\Longrightarrow}
\newcommand{\Tfrac}[2]{{\Ts\frac{#1}{#2}}}
\newcommand{\Tsum}{\mathop{\Ts\sum}\nolimits}

\newcommand{\wsup}{\mathop{\rm sup\vphantom{f}}\nolimits}

\renewcommand{\vec}{\mathop{\rm vec}\nolimits}
\newcommand{\vech}{\mathop{\rm vech}\nolimits}
\newcommand{\osk}{\mathrel{{\otimes}_{\!\!\!\!\!\!\!\!\!{{=}\atop{s}}}}}
\newcommand{\wto}{\mathrel{\mathsurround0em \mbox{$\longrightarrow$}%
  \llap{\settowidth{\SyW}{$\longrightarrow$}%             \protect in @{}
  \raisebox{-.15ex}{\makebox[\SyW]{\scriptsize\rm w}}}}}

\def\pth{\partial_{\theta}}
\def\pthj{\partial_{\theta_j}}
\def\px{\partial_{x}}
\def\pxi{\partial_{x_i}}
\def\pxx{\partial^2_{xx}}
\def\pxir{\partial^2_{x_ix_r}}
\def\tath{\iota_{\theta}}
\def\tathl{\iota_{\theta;l}}
\def\tathm{\iota_{\theta;m}}
\def\tth1{\tau_{\theta}}
\def\ptxo1{{( \px\tath )}^{-1}}
\def\ptth{\pth\tath}
\def\ptx{\px\tath}
\def\ptxx{\pxx\tath}
\def\dttx{|\!\det\px\tath\!|}
\newcommand{\pthI}[1]{\partial_{\theta_{#1}}}
\newcommand{\pxI}[1]{\partial_{x_{#1}}}

\newtheorem{Thm}{Theorem}%[section]   % STARTS each SECTION
\newtheorem{Prop}[Thm]{Proposition}
\newtheorem{Lem}[Thm]{Lemma}
\newtheorem{Rem}[Thm]{Remark}
\newtheorem{Cor}[Thm]{Corollary}
\newtheorem{Bez}[Thm]{Notation}
\newtheorem{Def}[Thm]{Definition}

\newtheorem{Exa}[Thm]{Example}

\numberwithin{equation}{section}
\numberwithin{Thm}{section}

\newcounter{ABCc}
\renewcommand{\theABCc}{\alph{ABCc}}
\newenvironment{ABC}{\begin{list}{%    VE LABELWIDTH 1em + LABELSEP 0.5em
  \rfia{\theABCc}}{\usecounter{ABCc} \topsep 0ex \partopsep 0ex \itemsep0ex
  \parsep=\parskip \leftmargin 0em \rightmargin 0em \itemindent=\parindent
  \listparindent=\parindent  \labelsep 0.5em \labelwidth 0.5em }}{\end{list}}
\newcommand{\seitref}{}
\newenvironment{proof}[1]
{\begin{trivlist}\item[]
\ifthenelse{\equal{#1}{ }}
   {{Proof:}}
   {{Proof\/ #1\/:}}
}
{\hfill  \qed\end{trivlist}}
\newcommand\unpubman{Unpublished manuscript}

 \journal{Journal of Multivariate Analysis}
\begin{document}
\begin{frontmatter}
%-----------------------------------------------------------------------------
\title{Fisher Information in Group-Type Models}
%-----------------------------------------------------------------------------
\author{Peter Ruckdeschel}
\ead{Peter.Ruckdeschel@itwm.fraunhofer.de}
\date{\today}
\address{Fraunhofer ITWM, Abt.\ Finanzmathematik,
         Fraunhofer-Platz 1, 67663 Kaiserslautern, Germany\\
     and TU Kaiserslautern, AG Statistik, FB.\ Mathematik,
         P.O.Box 3049, 67653 Kaiserslautern, Germany
}
%-----------------------------------------------------------------------------
\date{Received: date / Accepted: date}
%\maketitle
% -----------------------------------------------------------------------
\begin{abstract}
In proofs of $L_2$-differentiability, 
Lebesgue densities of a central distribution are often assumed right from 
the beginning. Generalizing  \citet[Theorem~4.2]{Hu:81}, we show that in the class of smooth
parametric group models these densities are in fact consequences of
a finite Fisher information of the model, provided a suitable representation
  of the latter is used. The proof uses  the notions of absolute continuity
in $k$ dimensions and weak differentiability. \\
As examples to which this theorem applies, we spell out a number of models
including a correlation model and the general multivariate location and
scale model.\\
As a consequence of this approach, we show that in  the (multivariate)
location scale model, finiteness of Fisher information as defined here
is in fact equivalent to $L_2$-differentiability and
 to a log-likelihood expansion giving local asymptotic normality of the model.\\
Paralleling Huber's proofs for  existence and uniqueness of a minimizer of
Fisher information to our situation, we get existence of a minimizer in
any weakly closed set ${\cal F}$ of central distributions $F$.
If, additionally to analogue assumptions to those of \citet{Hu:81},
 a certain identifiability condition for the transformation
 holds, we obtain uniqueness of the minimizer.
 This identifiability condition is satisfied in the multivariate location
 scale model.
\end{abstract}
\end{frontmatter}
% -----------------------------------------------------------------------
%\tableofcontents
% -----------------------------------------------------------------------
\section{Introduction}
\subsection{Motivation}
$L_2$-differentiability as introduced by LeCam and H\'ajek appears to
be the most suitable setup in which to derive such key properties as
\textit{local asymptotic normality\/} (LAN)  in local asymptotic parametric
statistics. 
In order to show this $L_2$-differentiability however, Lebesgue densities
of a central distribution are frequently assumed right from the beginning.
In this paper, we generalize \citet[Theorem~4.2]{Hu:81} from one-dimensional
location to a large class of parametric models, where these Lebesgue densities
are in fact a consequence of a finite Fisher information of the model,
provided a suitable definition of the latter is used. This definition may
then serve---again as in \citet{Hu:81}---as starting point for
minimizing Fisher information along suitable neighborhoods of the
model.\\
The framework in which this generalization holds covers smooth parametric
group models as to be found in \citet{BKRW:98}, but is valid even in a
somewhat more general setting: The idea is to link transformations in the
parameter space to transformations in the observation space. \\
The new definition of Fisher information then simply amounts to transferring
differentiation in the parameter space to
differentiation---in a weak sense---in the observation space.
This is actually done much in a Sobolev spirit, working with generalized derivatives.
\subsection{Organization of the Paper}
%
%This paper is organized as follows:
%
After an introduction to the setup of smooth parametric group models,
in section~\ref{setup}, we list the smoothness requirements for the
transformations and some notation needed for our theorem.
Before stating this theorem, in section~\ref{exasec} we first give a
number of examples to which
this theorem applies, the most general of which is the multivariate location and scale
model from Example~\ref{kls}.
Section~\ref{mtsec} provides the main result, Theorem~\ref{dasTHM}.
In section~\ref{exasec2}, we spell out the resulting Fisher information
in the examples of section~\ref{exasec}.
As announced in the motivation, in section~\ref{Cons}, culminating
in Proposition~\ref{equivprop}, we show that in  the (multivariate)
location-scale model finiteness of Fisher information is equivalent
to $L_2$-differentiability as well as to a LAN property.
Finally, in section~\ref{MinI} we generalize Huber's proofs for  existence
and uniqueness of a minimizer of
Fisher information to our situation.
The proofs are gathered in appendix section~\ref{proofs};
The proof of Theorem~\ref{dasTHM} makes use of the notions of absolute continuity
in $k$ dimensions and of weak differentiability.
Both are provided in an appendix in section~\ref{app}.
\begin{Rem}\rm\small The one-dimensional scale model, a particular case
of what is covered by this paper, has been spelt out separately, in a small
joint paper with Helmut Rieder, cf.\ \citet{Ru:Ri:10}.
\end{Rem}
\section{Setup}\label{setup}
\subsection{Notation}
$\B^k$ denotes the Borel $\sigma$-algebra on $\R^k$,
${\cal M}_1({\cal A})$ $[{\cal M}_s({\cal A})]$ the set of all probability [substochastic]
measures on some $\sigma$-algebra ${\cal A}$,
and for $\mu\in{\cal M}_1(\B)$, for $p\in[1,\infty]$, $L_p(\mu)$ is the set of all (equivalence classes of)
${\cal A}|\B$ measurable functions with $\Ew|X|^p<\infty$, resp.\ $\wsup_P|X|<\infty$.
%Correspondingly, ${\cal M}_{\SSs \leq 1}({\cal A})$ denotes the set of all substochastic measures $\mu$ on
%the $\sigma$--Algebra ${\cal A}$, i.e. all finite measures with $\mu(\Omega)\leq 1$.
$\Jc_A$ denotes the indicator function of the set $A$. $\EM_k$ is the $k$-dimensional unit matrix,
$\vec(A)$ is the operator casting a matrix to a vector, stacking the columns of $A$ over each other,
$\vech$ the operator casting the upper half of a quadratic matrix to a vector---including
the diagonal---and $A\otimes B$ the Kronecker product of matrices, and,
for $A,B\in \R^{k\times k}$, the symmetrized product
$A\osk B:=(AB+B^{\tau}A^{\tau})/2$.

For $l\in \N_0\cup \infty$ let ${\cal C}^{l}$ be the set of all
$l$ times continuously differentiable functions, where%
---if necessary---we specify domain and range in the
notation ${\cal C}^{l}(\mbox{domain},\mbox{range})$. Weak convergence of
measures $P_n\in{\cal M}_{1}(\Bq^k)$ to some measure $P\in{\cal M}_{1}(\Bq^k)$ is denoted by $P_n\wto P$.\\
Inequalities and intervals in $\R^k$ are denoted by the same symbols as in one dimension, meaning e.g.\
$l<r$ iff $l_i<r_i$,  for all $i=1,\ldots,k$,
and %for $l<r$,
$[l,r]:=\{x\in \R^k\,|\,l_i\leq x_i\leq r_i,  \quad \forall\;i=1,\ldots,k\}$.

Let $P_{\theta}\in{\cal M}_1(\B^k)$. $\R^k$ being Polish, regular conditional distributions are available,
and we may write $P_{\theta}(dx_1,\ldots,dx_k)$ as
\begin{equation}
P_{\theta}(dx_1,\ldots,dx_k)=\prod_{j=1}^{k-1}
P_{\theta;\,j|j+1:k}(dx_j|x_{j+1},\ldots, x_{k})\,\, P_{\theta;\,k}(dx_k)
\end{equation}
with $P_{\theta;\,k}$ the marginal of $X_k$ and $P_{\theta;\,j|j+1:k}$ a regular
conditional distribution of $X_{j}$, given
$X_{j+1}=x_{j+1},\ldots,X_k=x_k$. In the sequel, we write $y_{i:j}$ for the vector
$(y_i,\ldots,y_j)^{\tau}$. For a measure $G$ on ${\cal M}(\B^k)$ and a set of indices $J$ we write
$G_{J}$ to denote the joint marginal of $G$ for coordinates $i\in J$.
For $y \in \R^{k}$ define $y_{-i}:=y_{1:i-1;i+1:k}$,
and for $y \in \R^{k-1}$ and $x\in \R$ define the expression
$(x\!:\!y)_i:=(y_{1:i-1},x,y_{i:k-1})^{\tau}\in\R^k$.

\subsection{Model Definition}\label{modeldef}
For a fixed central distribution $F$ on $\B^k$, we consider a
statistical model  ${\cal P}\subset {\cal M}_1(\B^k)$ generated by a family
${\cal G}$ of diffeomorphisms $\tau:\R^k\to\R^k$ defined on the observation space.
Denote the inverse of $\tau$ by $\iota=\tau^{-1}$. This family is parametrized
by a $p$ dimensional parameter $\theta$, stemming from an open
parameter set $\Theta \subset \R^p$, and this induces the parametric model
\begin{equation}\label{modeldef}
{\cal P}=\{P_{\theta}\,|\;P_{\theta}=\tau_{\theta}(F),\quad\;\theta\in\Theta\}
\end{equation}
where $\tau_{\theta}(F)$ denotes the image measure under $\tau_{\theta}$,
$F\circ \iota_{\theta}$.%\\
\begin{Rem}\rm\small
In most examples, ${\cal G}$ will be a group, which is also the
formulation used in \citet[section~1.3]{Leh:83} and \citet[Ch.~4]{BKRW:98}.
These authors did not intend to generalize Fisher information, though, and
Example~\ref{corex} shows that for our purposes a group structure of
for the set ${\cal G}$ is not necessary.
\end{Rem}
%\\
%
%
\subsection{A Smooth Compactification of $\R^k$}
\noindent For reasons explained in Remark~\ref{rm41}, we introduce
the following compactification $\Rq^k$
 of $\R^k$:

\begin{Def}\label{compactifk}
  Let ${\cal C}^{l}([0,1]^k,\R)$, $l\in\N\cup\infty$ the space of all continuous real-valued functions on the domain $[0,1]^k$ which are differentiable
  $l$ times / arbitrarily often in $(0,1)^k$, and with existing one-sided derivatives on $\partial[0,1]^k$.
  We identify this space with functions on $\Rq^k$,
  using the isometry%\footnote{We deliberately denote this isometry by the same symbol as for $k=1$.}
  $\ell$
  \begin{equation}\label{ellk}
  \ell:[-\infty;\infty]^k\to[0,1]^k,\quad [\ell((x_j))]_i=\big[e^{x_i}/(e^{x_i}+1)\big]_i
  \end{equation}
  i.e.\ let
  \begin{equation}
    {\cal C}^{l}(\Rq^k,\R):={\cal C}^{l}([0,1]^k,\R)\!\circ\! \ell=\{\varphi\,|\,\varphi=\psi\!\circ\! \ell\quad\exists \psi \in {\cal C}^{l}([0,1]^k,\R)\}
  \end{equation}
\end{Def}
For later purposes we also note the inverse of $\ell$
  \begin{equation} \label{ell-1}
  \kappa:[0,1]^k\to [-\infty;\infty],\quad \kappa(y_1,\ldots,y_k)=\big(\log (y_j/(1-y_j))\big)_{j=1,\ldots,k}
  \end{equation}
In the same manor, unbounded, continuous functions are defined and denoted by
%  \begin{equation}\label{compactifkbar}
%$    {\cal C}^{l}(\Rq^k,\Rq):={\cal C}^{l}([0,1]^k,\Rq)\!\circ\! \ell,$,
$    {\cal C}^{l}(\Rq^k,\Rq^m)
%:=[{\cal C}^{l}(\Rq^k,\Rq)]^m
$.
%  \end{equation}

\begin{Rem}\label{D2Rem}\rm\small
\begin{ABC}
\item With this definition, $\Rq^k$ becomes a compact metric space.
\item Integrations along $\Rq^k$ are understood as
lifted onto $[0,1]^k$ by $\ell$, i.e.\
  $ \int_{\sRq^k}  f\, dP=\int_{[0,1]^k}
  f\!\circ\!\kappa\, d[\ell\!\circ\! P]$.
 \item The choice of $\ell$ resp.\ $\kappa$ is arbitrary to some extent,
       but satisfactory for our needs; in fact, we only
       have to impose $\ell\in{\cal C}^{\infty}(\R^k,\R^k)$,
       $\lim_{x\to-\infty}(\ell((x\!:\!y)_i))_i=0$, $\lim_{x\to\infty}(\ell((x\!:\!y)_i))_i=1$, for each $y\in\R^{k-1}$,
        $\ell$ strictly isotone in each coordinate,
       $|\ell'(x)D(\tth1(x))| \in L_2(F)$, or, for uniformity in
       ${\cal M}_1(\B^k)$, $\sup_x|\ell'(x)D(\tth1(x))| <\infty$.
\item For every $\varphi \in {\cal C}^{\infty}(\Rq,\R)$, the
limits $\lim_{x\to\pm\infty}\varphi(x)$ exist and
  $\lim_{x\to\pm\infty}\frac{d^l}{dx^l}\varphi(x)=0$ for $l\geq 0$, as
  is easily seen using the chain rule and by the fact that each summand
  arising in a derivative has at least a factor decaying as $\exp(-|x|)$.
  This also implies that there are functions
  $\tilde \varphi:\bar R\to\R$ which do not lie
  in ${\cal C}^{\infty}(\Rq,\R)$
  but which are in ${\cal C}^{\infty}(\R,\R)$,
  have existing $\lim_{x\to\pm\infty}\tilde \varphi(x)$, and
  for which
  $\lim_{x\to\pm\infty}\frac{d^k}{dx^k}\tilde\varphi(x)=0$ for $k\geq 0$:
  Take $1/(x^2+1)$, which has no exponentially decaying derivatives.
  \item Consequently,
  for all $\varphi\in {\cal C}^{\infty}(\Rq^k,\R)$,
  $\int |b\varphi'|\,d\lambda$ is finite for any bounded,
  measurable function $b$ and
  $\lim_{|x|\to\infty} \varphi(x)' |x|^k=0$ for all $k\in\N$,
  hence in particular is
        in $L_\infty(P)$ for every probability $P$ on $\B$.
\item If we allow for mass of $\ell\!\circ\! P$ in
$[0,1]^k\setminus(0,1)^k$---corresponding to measures in
${\cal M}_s(\B^k)$---the class ${\cal C}_c^{k}(\R,\R)$ of compactly supported functions in
${\cal C}^{k}(\R,\R)$ cannot distinguish any
measures $P_1\not=P_2$ on $\Bq^k$ coinciding on $\B^k$,
whereas ${\cal C}^{\infty}(\R^k,\R)$ is measure determining on $\Bq^k$.
\item
The measures $P_{\theta}$ arising in our model from subsection~\ref{modeldef}
are understood as members of ${\cal M}_1(\Bq^k)$, defining
$P_{\theta}(A)=P_{\theta}(A\cap \R^k)$ for $A\in \Bq^k$.
\end{ABC}
\end{Rem}
\subsection{Assumptions}
Throughout this paper, we make the following set of assumptions concerning the transformations $\tau$,
which are needed to link differentiation w.r.t.\ $\theta$ to differentiation w.r.t.\ $x$:
\begin{description}
\item[(I)]
$
P_{\theta_1}=P_{\theta_2}\qquad\iff\qquad\theta_1=\theta_2
$.
\item[(D)] %For $\lambda(dx)$--a.e.\ $x$,
$\theta\mapsto \tath(x)$ is differentiable with derivative
$\pth \tath(x)$.
\item[(Dk)] If $k>1$, %for $F(dx)$--a.e.\ $x$,
$x\mapsto \tath(x)$ is twice differentiable with second derivative
$\ptxx(x)$.
\item[(C1)] If $k=1$, %For $\lambda(dx)$--a.e.\ $x$,
$x \mapsto D$ is in ${\cal C}^1(\Rq,\Rq^p)$ and $x\mapsto e^{-|x|}D\!\circ\!\tth1(x)$ is in $L^{p}_{2}(F)$ with
\begin{equation}\label{D1def}
D=D_{\theta}^{(1)}(x)=\ptth(x)/\ptx(x)
\end{equation}
\item[(Ck)] If $k>1$, %for $F(dx)$--a.e.\ $x$,
$x\mapsto D$ is in ${\cal C}^1(\Rq^k,\Rq^{k\times p})$, $x\mapsto e^{-|x|}D\!\circ\!\tth1(x)$
is in $L^{k\times p}_{2}(F)$ and
$x\mapsto V\!\circ\!\tth1(x)$ is in $L^{p}_{2}(F)$
with \begin{eqnarray}
J&=&(J_{\theta}(x))_{i,j=1,\ldots,k}=(\ptxo1)_{i,j}(x)
\label{Jdef}\\
D&=&(D^{(k)}_{\theta}(x))_{{i\!=\!1\ldots k}\atop{j\!=\!1\ldots p}}= [J^{\tau} \pth \tath]_{i,j}(x)
\label{Ddef}\\
V&=&(V_{\theta}(x))_{j\!=\!1\ldots p}= \frac{[\Tsum_{i=1}^k \pxi (\dttx\, D_{i,j})- \pthj \dttx\,]_j}{\dttx} (x)
\label{Vdef}
\end{eqnarray}
%\item[(E)]
%$P_{\theta}(K)<1$ with
%\begin{equation}
%K:=\{D=0 \}\label{K-def}
%\end{equation}
\end{description}
\begin{Rem}\rm\small
%Recall Cramer's rule,
%\begin{equation}
%(A^{-1})_{j,i}=\frac{\det A^{i,j}}{\det A}, \label{Cramer}
%\end{equation}
%so that we have by Laplace's expansion of a determinant and  \eqref{Cramer}
Using
%\begin{eqnarray}
$\frac{\partial}{\partial A_{i,j}} \det A
%&=& \lim_{h\downarrow 0} \frac{1}{h}\left[\det(a_1,\ldots,a_{j-1},a_j+he_i,a_{j+1}\ldots,a_k)-\det A\right]= \nonumber\\
%&=&\det A^{i,j}
= (A^{-1})_{j,i}\det A %\label{ainv1}
$%\end{eqnarray}
%Also note that
and
%\begin{eqnarray}
$\frac{\partial}{\partial A_{k,l}} (A^{-1})_{i,j} %&=& \lim_{h\downarrow 0} \frac{1}{h}\left[\big((A+he_ke_l^{\tau})^{-1}\big)_{i,j}-(A^{-1})_{i,j}\right]= \nonumber\\
%&=&
- (A^{-1})_{i,k}(A^{-1})_{l,j}%\label{ainv2}
$
%\end{eqnarray}
%With \eqref{ainv1} and \eqref{ainv2} and using
and the chain rule of differentiation, one can show
%by straightforward calculations thats
\begin{eqnarray}
V_j&=&[{\dttx} ]^{-1}\big[\Tsum_{i=1}^k \{\pxi (\dttx\, D_{i,j}) \}- \pthj \dttx\big]=\nonumber \\
%&=&[{\dttx} ]^{-1}\big[\Tsum_{i=1}^k \{ D_{i,j} (\pxi \dttx)  + (\pxi D_{i,j})\dttx\}-%\nonumber\\
%&&\quad -
%\pthj \dttx \big]= \nonumber\\
%&\stackrel{\mbox{\tiny\eqref{ainv1}}}{=}&\Tsum_{i,l,m,r=1}^k \{ J_{l,i} \pthj\tathl   J_{r,m }\pxir\tathm   \} +\nonumber\\
%&&\quad + \Tsum_{i,l=1}^k \{[\pxi  J_{l,i}]  \pthj\tathl + J_{l,i} [\pxi \pthj\tathl]  \}- \nonumber\\
%&&\quad - [{\dttx} ]^{-1}\pthj \dttx=\nonumber\\
%&\stackrel{\mbox{\tiny\eqref{ainv2}}}{=}&\Tsum_{i,l,m,r=1}^k   (J_{l,i} J_{r,m} - J_{l,r}J_{m,i} ) \, \pxir\tathm \, \pthj\tathl
% + \nonumber \\
%&&\quad + \Tsum_{i,l=1}^k \{J_{l,i} [\pxi \pthj \tathl]  - J_{l,i} [\pxi \pthj \tathl]   \} =\nonumber\\
&=&
\Tsum_{i,l,r,m=1}^k   (J_{l,i}J_{r,m} - J_{l,r}J_{m,i}) \, \pxir\tathm \, \pthj\tathl ,\label{Vnummer}
\end{eqnarray}
which motivates requirement (Dk).
\end{Rem}
In the sequel we use these abbreviations:
\begin{Bez}\rm
The set $\{D=0 \}$ is denoted by $K$.
With $e_i$ the $i$-th canonical unit vector in $\R^k$
and some $a\in\R^p$ and $y\in\R^{k-1}$, define
\begin{eqnarray}
\label{Vadef}
V_{a}&:=&V^{\tau} a,\qquad
D_{a}:=Da,\qquad
D_{a;i}:=e_i^{\tau}D_a,\qquad
K_i:=\{e_i^{\tau}D=0\} \label{Kidef}
%\label{Edef}
%E_{a,i,y}&:=&\{x \in\R\,|\,D^{(k)}_{a;i}((x\!:\!y)_i)=0\}.
\end{eqnarray}
Also, for later purposes---c.f. \eqref{IdefkkF}---we introduce the functions
\begin{eqnarray}
\tilde D&=& (\tilde D^{(k)}_{\theta}(x))_{{i\!=\!1\ldots k}\atop{j\!=\!1\ldots p}}= [\pth \tath \circ \tau_{\theta}]_{i,j}(x),\quad %\label{tDdef}\\
\tilde V=(\tilde V_{\theta}(x))_{j\!=\!1\ldots p}=V_{\theta} \circ \tau_{\theta} \label{tVdef}\\
\tilde V_{a}&:=& \tilde V^{\tau} a,\qquad
\tilde D_{a}:=\tilde Da, \qquad %\label{tVadef}\\
\tilde D_{a;i}:= e_i^{\tau} \tilde Da,\qquad
\tilde K_i:=\{e_i^{\tau}\tilde D=0\} \label{tKidef}
\end{eqnarray}
Finally, if $F\ll \lambda^k$, we write $f_{\theta}$ for $f\!\circ\! \iota_{\theta}$, with $f$ a $\lambda^k$ density  of $F$.
\end{Bez}
We also introduce the following decomposition of $P_{\theta}$:
%We decompose $K^c$ disjointly according to the following scheme
%\begin{eqnarray}
%&&K_i=\{D^{(k)}_{i}\not=0\},\qquad K_1^0:=K_1,\quad K_i^0=K_i\cup \bigcup_{l\leq i}(K_l^0)^c,\label{kdec}\\
%&&K^c=\dot{\bigcup_{1\leq i\leq k}} (K_i^0)^c,
%\end{eqnarray}
%and then decompose $P_{\theta}$ according to
\begin{equation}
P_{\theta}:=P^{(0)}_{\theta}+\bar P^{(0)}_{\theta},\qquad \bar P^{(0)}_{\theta}(\cdot):=P_{\theta}(\cdot\cap K).\label{p0def}
\end{equation}
%
%
%We will give it in two versions, in order to be able to cover the case that the distribution of
%some $k-s$--dimensional projection of $X$ does not admit a Lebesgue density, but that the projection
%of $X$ onto the corresponding orthogonal complement actually is sufficient for the parametric problem, and that this
%latter projection does admit a a Lebesgue density. This gives rise to assumption (Es) and will be relevant in the correlation model from
%Example~\ref{corex}. In most cases however, we will work with assumption (E), which
%is much more easily assessible.
%
%
\section{Examples}\label{exasec}
For the following seven popular examples we spell out the transformations $\tau_\theta(x)$ and the respective
parameter space and verify the assumptions from the preceding section.
\begin{Exa}[one-dim.\ location]\label{1dimlok}\sf $\tau_{\theta}(x):=x+\theta$, $\theta\in\Theta_1=\R$, $p=k=1$.\\
 For each $\theta \in \Theta_1$, $\tau_{\theta}(\cdot)$ is  a diffeomorphism; assumptions (I), (D)
 and (C1) are satisfied---$\ptth=-1$, $\ptx=1$, $D(x)=-1$, $K=\emptyset$---any observation $x$ is informative for this problem.
\end{Exa}
\begin{Exa}[$k$-dim.\ location, $k>1$]\label{multilok}\sf $\tau_{\theta}(x):=x+\theta$, $\theta\in\Theta_2=\R^k$ $p=k$.\\
 For each $\theta \in \Theta_2$, $\tau_{\theta}(\cdot)$ is  a diffeomorphism; assumptions (I), (D), (Dk),
 and (Ck) are satisfied--- $\ptxx=0$, $\ptth=-\ptx=D=-\EM_k$, $V=0$, $K=\emptyset$---any observation $x$ carries information for this problem.
\end{Exa}
\begin{Exa}[one-dim.\ scale]\label{onescale}\sf $\tau_{\theta}(x):=\theta x$, $\theta\in\Theta_3=\R_{>0}$, $p=k=1$.\\
 For each $\theta \in \Theta_3$, $\tau_{\theta}(\cdot)$ is  a diffeomorphism; assumptions (I), (D)
 and (C1) are satisfied--- $\ptth=-x\!/\theta^2$, $\ptx=1\!/\theta$, $D(x)=-x\!/\theta$.
 Thus $K=\{0\}$, hence the point $x=0$ is not informative for this problem,
 and any $x\not=0$ is.
\end{Exa}
\begin{Exa}[one-dim.\ loc.\ and scale]\label{onels}\sf $\tau_{\theta}(x):=\theta_2 x+\theta_1$,
% with parameter set
$\theta\in \Theta_4=\Theta_1\times\Theta_3$, $k=1$, $p=2$.\\
 For each $\theta\in\Theta_4$, $\tau_{\theta}(\cdot)$ is  a diffeomorphism; assumptions (I), (D)
 and (C1) are satisfied---consider $\ptth=-(\frac{1}{\theta_2};\frac{x-\theta_1}{\theta_2^2})^{\tau}$,
 $\ptx=\frac{1}{\theta_2}$, $D(x)=-(1;\frac{x-\theta_1}{\theta_2})$, $K=\emptyset$---any observation $x$ 
carries information for this problem.
\end{Exa}
\begin{Exa}[correlation, $k=2$; $p=1$]\label{corex}\sf To $\sigma_1,\sigma_2>0$ known let
$\theta\in\Theta_5=(-1;1)$
\begin{equation}\label{cortaudef}
\tau_{\theta}:\R^2\to\R^2,\quad
x\mapsto\tau_{\theta}(x):=J_{\theta}x,\qquad J_{\theta}= %(1-\theta^2)^{-1/2}
\left( \begin{array}
 {cc}\sigma_1(1-\theta^2)^{\frac{1}{2}}&\theta \sigma_1\\0&\sigma_2
 \end{array}
%\begin{array}
%{cc}\sigma^{-1}_1& -\theta\sigma_2^{-1}\\0&(1-\theta^2)^{1/2}\sigma_2^{-1}
%\end{array}
\right)x;
\end{equation}
 In contrast to all other examples considered here, this family does not form a group;
this may easily be seen, as $J_\theta^{-1}$ does not admit a representation according to \eqref{cortaudef}.
For each $\theta\in \Theta_5$, $\tau_{\theta}(\cdot)$ is  a diffeomorphism; assumptions (I), (D), (Dk),
 and (Ck) are clearly satisfied---$\ptxx=0$, $V=0$, $\ptth=(1-\theta^2)^{-\frac{3}{2}}(x_1\theta\sigma_1^{-1}-x_2\sigma_2^{-1},0)^{\tau}$,
 $$\ptx=(1-\theta^2)^{-\frac{1}{2}}\left(
 \begin{array}
 {cc}\sigma_1^{-1}&-\theta \sigma_2^{-1}\\ 0&(1-\theta^2)/\sigma_2
 \end{array}\right),%\qquad
% J=\left(
% \begin{array}
% {cc}\sigma_1(1-\theta^2)^{\frac{1}{2}}&\theta \sigma_1\\0&\sigma_2
% \end{array}\right),
$$
 $D=([\theta x_1-\frac{\sigma_1}{\sigma_2} x_2]/(1-\theta^2),0)^{\tau}$.
  As $K=\{x \in\R^2\,|\, \exists \rho \in \R \;:\; x=\rho (\sigma_1;\theta\sigma_2)^{\tau}\}$, % each $x\in K^c$,
% carries information about $\theta$; so
 $P(K)<1$ holds, as long as ${\rm supp}(P_{\theta})$ is not contained in the line
 $\{\rho (1;{\theta\sigma_2}/{\sigma_1})^{\tau},\;\rho \in\R\}$ or equivalently, as long as
 ${\rm supp}(F)\not \subset
 \{ \rho((1-\theta^2)^{\frac{1}{2}};\theta)^{\tau},\;\rho \in\R\}$.% but as we will see, there will be at least
\end{Exa}
\begin{Exa}[$k$-dim.\ scale, $k>1$]\label{multscal}\sf $\tau_{\theta}(x):=\theta x$, defined for %the  parameter set
$\Theta_6=\{S\in \R^{k\times k} \,|\, S=S^{\tau}\succ 0\}$, $p={k+1\choose 2}$. The symmetry restriction is imposed on
$\R^{k\times k}$, allowing only for symmetric variations in the parameter.\\
 Again, for each $\theta\in\Theta_6$, $\tau_{\theta}(\cdot)$ is  a diffeomorphism; assumptions (I), (D), (Dk),
 and (Ck) are satisfied---$\ptxx=0$, $\pxI{i}\tathl=(\theta^{-1})_{l,i}$,
 \begin{eqnarray*}
 \pthI{i_1,\!i_2}\tathl&=&-\frac{1}{2}[(\theta^{-1})_{l,i_1}(\theta^{-1}x)_{i_2}+(\theta^{-1})_{l,i_2}(\theta^{-1}x)_{i_1}],\\
 D&=&D_{i;j_1,\!j_2}=\frac{1}{2}
 [(\EM\otimes \theta^{-1}x)_{i,j_1,\!j_2}+(\EM\otimes \theta^{-1}x)_{i,j_2,\!j_1}],\qquad V=0.
 \end{eqnarray*}
%As can be seen considering the sets $A^{-1}U_i$ for $A\in {\rm GL}(k)$ and
%$U_i=\{x\in\R^k\,|,x=(y_{1:i-1},x_i,y_{i+1:k}),\;x_i=0\}$,
%up to a $\lambda^{k-1}$-null set of values $y$, these sets only consist in isolated points;
%but, a
For each symmetric matrix $a\in {\rm GL}(k)$, we have $D(x)a=\theta^{-1}a\theta^{-1}x$; $K=\{0\}$---any 
observation $x\not=0$ carries information for this problem.
\end{Exa}
\begin{Exa}[$k$-dim.\ location and scale, $k>1$]\label{kls}\sf $\tau_{\theta}(x):=\theta_2x+\theta_1$,  for
$\theta\in\Theta_7=\R^k\times\Theta_6$, $p=k+{k+1\choose 2}=k(k+3)\!/2$.\\
 For each $\theta\in\Theta_7$, $\tau_{\theta}(\cdot)$ is  a diffeomorphism; assumptions (I), (D)
 (Dk), and (Ck) are satisfied---$\ptxx=0$, $V=0$, $\ptx=\theta_2^{-1}$;
 splitting off the indices for the parametric dimensions into the location part [a single index] and the scale part [a double index],
 we get
 \begin{eqnarray*}
 \pthI{i}\tathl&=&-(\theta_2^{-1})_{i;l},\qquad\\
   \pthI{i_1,\!i_2}\tathl&=&-\frac{1}{2}[(\theta_2^{-1})_{l,i_1}\big(\theta_2^{-1}(x-\theta_1)\big)_{i_2}+
   (\theta_2^{-1})_{l,i_2}\big(\theta_2^{-1}(x-\theta_1)\big)_{i_1}]; \\
D_{i,l}&=&-\EM_{i;l},\\
   D_{i_1,\!i_2,l}&=&-\frac{1}{2}[\big(\EM\otimes\theta^{-1}_2(x-\theta_1)\big)_{l,i_1,\!i_2}+
   \big(\EM\otimes\theta_2^{-1}(x-\theta_1)\big)_{l,i_2,\!i_1}].
\end{eqnarray*}
Just as in Example~\ref{multilok}, any observation $x$ carries information for this problem.
\end{Exa}
\section{Main Theorem} \label{mtsec}
%
%With these preparations we may now state the main result,
%generalizing \citet{Hu:81}:
In \citet[Definition~4.1 and Theorem~4.2]{Hu:81}, we find a result
on the Fisher information in the one dimensional location case which is central
for the famous minimax M estimator result of \citet{Hu:64}.
The idea is to express Fisher information as a supremum, i.e.\
\begin{equation} \label{Ihudef}
{\cal I}(F):=\sup \Big\{\frac{\left(\int \varphi'\,dF\right)^2}%
{\int \varphi^2 \,dF}\, \Big|\qquad 0\stackrel{[F]}{\not=} \varphi\in{\cal D}_1 \Big\}.
\end{equation}
With this definition, \citet[Thm~4.2]{Hu:81} achieves a representation of
Fisher information without assuming densities of the central
distribution: ${\cal I}(F)$ is finite iff $F$ is a.c.\ with a.c.\
Lebesgue density $f$ such that $\int (f'/f)^2 f\,dx
<\infty$, which in this case is just ${\cal I}(F)$.
%

%
%\subsection{A Larger Set of Test Functions}\label{LargeSubclass}
%
\begin{Rem}\label{rm41}\rm\small
\begin{ABC}
\item The proof in \cite{Hu:81} is credited to T.~Liggett and
is based on Sobolev-type ideas; we take these up to generalize
the result to more general models and higher dimensions.
\item The set ${\cal D}_{1}$ in \eqref{Ihudef} plays the r\^ole
of a set of test functions as in the theory generalized functions,
compare \citet[Ch.~6]{Ru:91}. In the cited reference, Huber uses
${\cal D}_{1}={\cal C}_c^1(\R ,\R)$, the subset
 of compactly supported functions in ${\cal C}^1(\R,\R)$.
In the proof later, we will need that the sets
\begin{equation} \nonumber
{\cal D}_{D;i,j}:=\{D_{a;j}\,\partial_{x_{i}}\phi\,|\,\phi\in{\cal D}_{k},\;a\in\R^p\}
\end{equation}
are dense in $L_2(P_\theta^{(j)})$. Contrary to the one-dimensional location case, 
for ${\cal D}_{k}={\cal C}_c^1(\R^k ,\R)$ and general
$D_{a;j}$, we did not succeed to prove this; nor can we work
with ${\cal D}_{k}={\cal C}_{c1}(\R^k ,\R)$, the set of continuously
differentiable functions with compactly supported derivatives, 
as used for the one dimensional scale model in \citet[Lem.~A.1]{Ru:Ri:10}:
The crucial approximation of the constant function $1$ by 
functions $\phi\in {\cal C}_{c1}(\R^k ,\R)$, with $|\phi|\leq 1$, 
$|D_{a;j}\,\partial_{x_{i}}\phi|\leq 1$, and $|D_{a;j}\,\partial_{x_{i}}\phi|\to 0$
pointwise, fails for functions $D_{a;j}$  growing
faster than $|x|$ for large $|x|$.
Hence, instead we use the larger set
${\cal C}^{\infty}(\Rq^k,\R)$ from Definition~\ref{compactifk}.
\end{ABC}
\end{Rem}
%\subsection{Statement of the Theorem}
%%%%%%%%%%%%%%%%%%%%%%%%%%%%%%%%%%%%%%%%%%%%%%%%%%%%%%%%%%%%
%neue Definition
%%%%%%%%%%%%%%%%%%%%%%%%%%%%%%%%%%%%%%%%%%%%%%%%%%%%%%%%%%%%
%
%
\begin{Def}\label{Itdef}
In model~${\cal P}$ from \eqref{modeldef}, assume
assume (I) and (D). Let $a\in\R^p$, $|a|=1$.
\begin{description}
\item[\textbf{$\boldmath{k=1}$:}]
Assume (C1). Let ${\cal D}_{1}={\cal C}(\Rq^1,\R)$, $D$  from \eqref{D1def}.
Then for $\theta\in\Theta$ we define
\begin{equation}\label{Idefk1}
{\cal I}_{\theta}(F;a):=\sup\Big\{\frac{\Big(\int [\varphi' D_a]\, dP_{\theta} \Big)^2}
                                 {\int \varphi^2 \, dP_{\theta} } \;\Big|\;
\qquad 0\stackrel{[P_{\theta}]}{\not =} \varphi \in{\cal D}_1 \Big\},
\end{equation}
\item[\textbf{$\boldmath{ k>1}$:}]
Assume (Dk) and (Ck). Let ${\cal D}_{k}={\cal C}(\Rq^k,\R)$, $D$ and $V$ from \eqref{Ddef} and \eqref{Vdef}.
Then for $\theta\in\Theta$ we define
\begin{equation}\label{Idefkk}
{\cal I}_{\theta}(F;a):=
\sup \Big\{\frac{\Big(\int [\nabla\varphi^{\tau} D_a + \varphi V_a] \, dP_{\theta} \Big)^2 }
          {\int \varphi^2 \, dP_{\theta}}  \;\Big|\;
\qquad 0\stackrel{[P_{\theta}]}{\not =} \varphi
\in{\cal D}_k\Big\}.
\end{equation}
\end{description}
\end{Def}
\begin{Rem}\rm\small
\begin{ABC}
\item As $\tau_{\theta}$, resp.\ $\iota_{\theta}$ map ${\cal D}_k$ onto itself, we may use the identification $\psi=\varphi\circ\tau_{\theta}$
to see that by the transformation formula
\begin{equation}\label{IdefkkF}
{\cal I}_{\theta}(F;a):=
\sup \Big\{\frac{\Big(\int [\nabla\psi^{\tau} \tilde D_a + \psi\circ\iota_\theta \tilde V_a] \, dF \Big)^2 }
          {\int \psi^2 \, dF}  \;\Big|\;
\qquad 0\stackrel{[P_{\theta}]}{\not =} \psi \in{\cal D}_k\Big\}.
\end{equation}
\item
In particular, the transformation formula $\int \rho(x) \,P_{\theta}(dx)=
\int \rho\!\circ\!\tth1\,dF,$ entails that
except for the correlation model of Example~\ref{corex}, finiteness
of the Fisher information for one $\theta\in \Theta$ implies finiteness for every $\theta\in \Theta$: Indeed, considering
$D_{\theta}^{(k)}\!\circ\! \tth1$ in all these models, we see that in every case, $D_{\theta}^{(k)}\!\circ\! \tth1=D_{\rm\Ss id}^{(k)}$,
where we write $\rm id$ referring to the parameter-value $\theta$ yielding $\tath={\rm id}$, while at the same time $V=0$.\\
So in fact we could define \textit{the} Fisher information of $F$ for one reference parameter,
and its finiteness then entails finiteness in the whole parametric model.
\item In general, finiteness will however depend on the actual parameter value,
which is why we define Fisher information at $F$ with reference to $\theta$, notationally transparent
as ${\cal I}_{\theta}(F;a)$.
\end{ABC}
\end{Rem}
%
%%%%%%%%%%%%%%%%%%%%%%%%%%%%%%%%%%%%%%%%%%%%%%%%%%%%%%%%%%%%
%Das Haupttheorem
%%%%%%%%%%%%%%%%%%%%%%%%%%%%%%%%%%%%%%%%%%%%%%%%%%%%%%%%%%%%
%
With Definition~\ref{Itdef} we generalize \citet[Thm.~4.2]{Hu:81} to
\begin{Thm}\label{dasTHM}
In model ${\cal P}$ from \eqref{modeldef} assume that for some fixed $\theta\in\Theta$, (I), and, if $k=1$,
(D), and (C1),  resp., if $k>1$, (Dk) and
(Ck) hold. Then (the sets of) statements (i) and (ii) are equivalent:
\renewcommand{\labelenumi}{(\roman{enumi})}
\begin{enumerate}
\item
$\sup_{a:\,|a|=1}{\cal I}_{\theta}(F;a)<\infty$
\item
\begin{enumerate}
\item $F$ admits a $\lambda^k$ density $f$ on $\tath(K^c)$.
\item For every $a\in\R^p$, and $i=1,\ldots,k$
$$\lim_{|x|\to \infty} [f_{\theta}\, \dttx\,D_{a;i}]((x\!:\!y)_i)=0$$
\item For every $a\in\R^p$, and $i=1,\ldots,k$
$f_{\theta}\, \dttx\,D_{a;i}$ is a.c.\ in $k$ dimensions in the sense of Definition~\ref{ACkdim}.
\item For every $a\in\R^p$ and $1\!\leq \!i\!\leq \!k$, $[\frac{\pxI{i} (\dttx\, D_{a;i})}{\dttx} +
\frac{D_{a;i} \pxI{i} f_{\theta}}{f_{\theta}}] \in L_2(P_{\theta})$.
\end{enumerate}
\end{enumerate}
If (i) resp.\ (ii) holds,  ${\cal I}_{\theta}(F;a)=a^{\tau}{\cal I}_{\theta}(F)a$ with
\begin{equation}\label{ass1}
{\cal I}_{\theta}=\int \Lambda_{\theta}\Lambda_{\theta}^{\tau}\,dP_{\theta},\qquad
\Lambda_{\theta}=(f'\!/\!f)\!\!\circ\!\! \tath\,\, \pth \tath+\frac{\pth \dttx}{\dttx},
\end{equation}
respectively
\begin{equation}\label{ass2}
\Lambda_{\theta}=\pth p_{\theta}\!/p_{\theta} \qquad
\mbox{ with }p_{\theta}=f_{\theta}\,\dttx.
\end{equation}
\end{Thm}
\begin{Rem}\rm\small
\begin{ABC}
\item Theorem~\ref{dasTHM} also covers model~\ref{1dimlok}; however,
it uses ${\cal D}_1={\cal C}^\infty(\Rq,\R)$ instead of ${\cal C}_c^1(\R,\R)$, hence, as
${\cal C}_c^1(\R,\R)\subset {\cal C}^\infty(\Rq,\R)$, finiteness of Fisher information in Huber's
definition formally is weaker than ours, so formally our implication (ii)$\also$(i) is harder,
(i)$\also$(ii) easier than his.
\item As a consequence of using ${\cal D}_1={\cal C}^\infty(\Rq,\R)$, we need (ii)(b), which does not
show up in the corresponding Theorems
\citet{Hu:81}(one-dim.\ location).
\item In Theorem~\ref{dasTHM}, $F$ may have $\lambda^k$ singular parts on $\tath K$.
But if so, then by Corollary~\ref{clt=0} necessarily, $\Lambda_{\theta}=0$ there.
This means that these parts do not contribute any information.
\item Closedness of a.c.\ functions under products \citep[7.2~Prob.4]{Du:02} entails that
under the assumptions of Theorem~\ref{dasTHM}, whenever the map $x\mapsto [D_{a;i}p_{\theta}]((x\!:\!y)_i)$ is a.c.\
on some interval $[c,d]$ where $D_{a;i}\not=0$, so is $p_{\theta}$.
\end{ABC}
\end{Rem}
%
%%%%%%%%%%%%%%%%%%%%%%%%%%%%%%%%%%%%%%%%%%%%%%%%%%%%%%%
%
% erst den Skalenfall
%
%%%%%%%%%%%%%%%%%%%%%%%%%%%%%%%%%%%%%%%%%%%%%%%%%%%%%%%
%
%
%%
%
%
%\begin{Cor}
%\end{Cor}
%
%
\section{Fisher information in Examples}\label{exasec2}
In this section we specify the terms $\Lambda_{\theta}$ and ${\cal I}_{\theta}(F;a)$, as well as the
quadratic form in $a$, ${\cal I}_{\theta}(F)={\cal I}_{\theta}$, for Examples~\ref{1dimlok} to
\ref{kls}. In the sequel, $\Lambda_f(x):=-\px f \!/\!f$
\begin{Exa}[one-dim.\ location]\sf $\Lambda_{\theta}(x):=\Lambda_f(x-\theta)$,
${\cal I}_{\theta}={\cal I}_0=\int \Lambda_f^2 \,dF$. The supremal definition of ${\cal I}(F)$ is \eqref{Ihudef}, but with
${\cal D}_1={\cal C}^\infty(\Rq,\R)$.
\end{Exa}
\begin{Exa}[$k$-dim.\ location, $k>1$]\sf $\Lambda_{\theta}(x):=\Lambda_f(x-\theta)$,
${\cal I}_{\theta}={\cal I}_0=\int \Lambda_f\Lambda_f^{\tau} \,dF.$, ${\cal I}_{\theta}(F;a)=a^{\tau}{\cal I}_0a$.
The supremal definition of ${\cal I}(F)$ is
\begin{equation}
{\cal I}_{0}(F;a):=\sup\Big\{\frac{\left(\int \nabla \varphi^{\tau} a\,\,dF\right)^2}{\int
\varphi^2 \,dF} \;\Big|\qquad 0\stackrel{[F]}{\not=} \varphi\in{\cal
D}_k \Big\}
\end{equation}

\end{Exa}
\begin{Exa}[one-dim.\ scale]\sf $\Lambda_{\theta}(x):=\frac{1}{\theta} [(x/\theta)\Lambda_f(x/\theta)+1]$,
${\cal I}_{\theta}=\frac{1}{\theta^2}{\cal I}_1=\frac{1}{\theta^2} \int (x\Lambda_f-1)^2 \,dF
=\frac{1}{\theta^2} (\int x^2\Lambda_f^2 \,dF -1).$
The supremal definition of ${\cal I}(F)$ is
\begin{equation}
{\cal I}_{1}(F):=\sup\Big\{\frac{\left(\int x \varphi'(x)\,F(dx)\right)^2}{\int
\varphi^2 \,dF} \;\Big|\qquad 0\stackrel{[F]}{\not=} \varphi\in{\cal
D}_1\Big\}
\end{equation}
\end{Exa}
\begin{Exa}[one-dim.\ location and scale]\sf
$$\Lambda_{\theta}(x):=\frac{1}{\theta_2} \left(%\begin{array}
%{c}
\Lambda_f(\frac{x-\theta_1}{\theta_2}),\;
(\frac{x-\theta_1}{\theta_2})\Lambda_f(\frac{x-\theta_1}{\theta_2})+1%
%\end{array}
\right)^\tau,
$$
$${\cal I}_{\theta}(x):=\frac{1}{\theta^2_2}{\cal I}_{0;1}(x)= \frac{1}{\theta^2_2}\left(\begin{array}
{cc}\int \Lambda_f^2 \,dF&\int x\Lambda_f^2 \,dF\\
\int x\Lambda_f^2 \,dF& \int (x\Lambda_f-1)^2 \,dF
\end{array}\right),$$
and ${\cal I}_{\theta}(F;a)=a^{\tau}{\cal I}_0a/\theta_2$.
With $a=(a_l,a_s)^{\tau}$, the supremal definition of ${\cal I}(F)$ is
\begin{equation}
{\cal I}_{e_2}(F;a):=\sup\Big\{\frac{\left(\int (a_l+a_sx) \varphi'(x)\,F(dx)\right)^2}{\int
\varphi^2 \,dF} \;\Big|\qquad 0\stackrel{[F]}{\not=} \varphi\in{\cal
D}_1\Big\}
\end{equation}
\end{Exa}
\begin{Exa}[correlation, $k=2$; $p=1$]\label{corex2}\sf
%Let $F_{1|2}$ the first marginal of $F$ and $f_1:= F_1'$;
%then
%
\begin{equation} \label{F12-1}
\sigma_2\sigma_1 \sqrt{1-\theta^2}\, P_{\theta}(dx_1,dx_2)=f(\frac{x_1/\sigma_1-\theta x_2/\sigma_2}{\sqrt{1-\theta^2}},\frac{x_2}{\sigma_2})
 \lambda^2(dx_1,dx_2)
\end{equation}
or with $ f=f_{1|2}f_2$
\begin{equation} \label{f12-1}
\sigma_1 \sqrt{1-\theta^2} \, p_{\theta;1|2}(x_1,x_2)= f_{1|2}(\frac{x_1/\sigma_1-\theta x_2/\sigma_2}{\sqrt{1-\theta^2}}),\qquad
\sigma_2 p_{\theta;2}(x_1,x_2)=f_2(\frac{x_2}{\sigma_2})
\end{equation}
and
\begin{equation} \label{f12-2}
(1-\theta^2) \Lambda_{\theta}(x_1,x_2)= \frac{f_{1|2}'}{f_{1|2}}(\frac{x_1/\sigma_1-\theta x_2/\sigma_2}{\sqrt{1-\theta^2}})
\frac{\theta x_1/\sigma_1 - x_2/\sigma_2}{\sqrt{1-\theta^2}}   +\theta ,\quad {\cal I}_{\theta}=\int\!\! \Lambda^2_{\theta}\,dP_{\theta}
\end{equation}
The supremal definition of ${\cal I}(F)$ is
\begin{equation}
{\cal I}_{\theta}(F):=\sup\Big\{\frac{\left(\int [{\theta x_1- \sqrt{1-\theta^2}\, x_2}](\partial_{x_1} \varphi)(x_1,x_2)\,F(dx_1,dx_2)\right)^2}{({1-\theta^2})^2\int
\varphi^2 \,dF} \;\Big|\qquad 0\stackrel{[F]}{\not=} \varphi\in{\cal
D}_2\Big\}
\end{equation}
\end{Exa}
\begin{Exa}[$k$-dim.\ scale, $k>1$]\sf
We give both $\vech$ expressions and
 matrix  expressions, using symmetrized Kronecker products.
%Furthermore we define the pointwise product of $\R^{k\times k}$--matrices $A\times B:=\sum_{i,j}
%A_{i,j}B_{i,j}=\tr AB^{\tau}$.
We start with unsymmetrized versions.
\begin{eqnarray*}
\Lambda^0_{\theta}(x)&=&\theta^{-1}\Lambda_{\sEM_k}(\theta^{-1}x)),\qquad
\Lambda^0_{\sEM_k}(x)=\Lambda_f(x)x^{\tau}-\EM_k,\\
\Lambda_{\theta}(x)&=&\frac{1}{2} [\Lambda^0_{\theta}(x)+\Lambda^0_{\theta}(x)^{\tau}],\qquad
\Lambda^v_{\theta}(x)=\vech[\Lambda_{\theta}(x)],
\end{eqnarray*}
This can also be written as
$
\Lambda^v_{\theta}(x)=\vech[\theta^{-1}\osk\Lambda_{\sEM_k}(x)]
$. %\\[-8ex]
In matrix notation this yields
$$
{\cal I}_{\theta}=((\theta^{-1}\otimes\theta^{-1})\osk[\int (\Lambda_f X^{\tau}-\EM_k)^{\otimes 2}\,dF],
$$
in vector notation ${\cal I}_{\theta}= \int\Lambda^v_{\theta}(\Lambda^v_{\theta})^{\tau}\,dF.$
Working with $a=a^{\tau}\in\R^{k \times k}$, we get
$$
\vech(a)^{\tau}\Lambda^v_{\theta}(x)=\Lambda_f(\theta^{-1}x)^{\tau}\theta^{-1}a\theta^{-1}x-\tr(\theta^{-1}a),
$$
$$
{\cal I}_{\theta}(F,a)=\int (\Lambda_f(y)^{\tau}\theta^{-1}ay-\tr(\theta^{-1}a))^2\,F(dy)
$$
For symmetric $a$, the supremal definition of ${\cal I}(F)$ is
\begin{equation} \label{Ikls}
{\cal I}_{\theta}(F,a):=\Big\{\sup\frac{\left(\int \nabla \varphi(x)^{\tau}  \theta^{-1}ax\,F(dx)\right)^2}{\int
\varphi^2 \,dF} \;\Big|\qquad 0\stackrel{[F]}{\not=} \varphi\in{\cal
D}_k\Big\}
\end{equation}
\end{Exa}
\begin{Exa}[$k$-dim.\ location and scale, $k>1$]\sf
Partitioning $\Lambda$ into a location block ($\rm l$) and a scale block ($\rm s$), we get
\begin{eqnarray*}
\Lambda_{{\rm \Ss l},\theta_1,\theta_2}(x)&=&\theta_2^{-1}\Lambda_{{\rm \Ss l},0,\sEM_k}(\theta_2^{-1}(x-\theta_1)), \qquad
\Lambda_{{\rm \Ss l},0,\sEM_k}(x)=\Lambda_f(x)\\
\Lambda_{{\rm \Ss s},\theta_1,\theta_2}(x)&=&\theta_2^{-1}\Lambda_{{\rm \Ss s},0,\sEM_k}(\theta_2^{-1}(x-\theta_1)),\qquad
\Lambda_{{\rm \Ss s},0,\sEM_k}^{0}(x)=\Lambda_f(x)x^{\tau}-\EM_k\\
\Lambda_{{\rm \Ss s},0,\sEM_k}(x)&=&(\Lambda_{{\rm \Ss s},0,\sEM_k}^{0}(x)+\Lambda_{{\rm \Ss s},0,\sEM_k}^{0}(x)^{\tau})/2,\qquad
\Lambda_{{\rm \Ss s},0,\sEM_k}^v(x)=\vech(\Lambda_{{\rm \Ss s},0,\sEM_k}(x))
\end{eqnarray*}
$${\cal I}_{\theta}=\left(\begin{array}
{cc} {\cal I}_{{\rm \Ss l},{\rm \Ss l},\theta}&{\cal I}_{{\rm \Ss l},{\rm \Ss s},\theta}\\
{\cal I}^{\tau}_{{\rm \Ss l},{\rm \Ss s},\theta}&{\cal I}_{{\rm \Ss s},{\rm \Ss s},\theta}\\
\end{array}
\right)
$$
with
\begin{eqnarray*}
{\cal I}_{{\rm \Ss l},{\rm \Ss l},\theta}&=&\theta_{2}^{-1}[\int \Lambda_f\Lambda_f^{\tau}\,dF]\theta_{2}^{-1},\\
{\cal I}_{{\rm \Ss l},{\rm \Ss s},\theta}&=&\theta_{2}^{-1}\big[\int \Lambda_f\vech[\theta_{2}^{-1}\osk(\Lambda_f\,X^{\tau}-\EM_k)]^{\tau} \,dF\big]\\
{\cal I}_{{\rm \Ss s},{\rm \Ss s},\theta}&=&\int \vech[\theta_{2}^{-1}\osk(\Lambda_f\,X^{\tau}-\EM_k)]\vech[\theta_{2}^{-1}\osk(\Lambda_f\,X^{\tau}-\EM_k)]^{\tau}\,dF
\end{eqnarray*}
Working with $a=(a_l^{\tau};\vech(a_s)^{\tau})^{\tau}$, $a_l\in\R^k$ $a_s=a_s^{\tau}\in \R^{k\times k}$, we get
\begin{small}
$$
\vech(a)^{\tau}\Lambda^v_{\theta}(x)=a_l^{\tau}\theta_2^{-1}\Lambda_f(\theta_2^{-1}(x-\theta_1))+
 \Lambda_f(\theta_2^{-1}(x-\theta_1))^{\tau}\theta_2^{-1}a_s\theta_2^{-1}(x-\theta_1)-\tr(\theta_2^{-1}a_s),
$$
\end{small}
$$
{\cal I}_{\theta}(F,a)=\int (a_l^{\tau}\theta_2^{-1}\Lambda_f(y)+
 \Lambda_f(y)^{\tau}\theta_2^{-1}a_sy-\tr(\theta_2^{-1}a_s))^2\,F(dy)
$$
The supremal definition of ${\cal I}(F)$ is
\begin{equation}
{\cal I}_{\theta}(F,a):=\sup\Big\{\frac{\left(\int \nabla \varphi(x)^{\tau} \theta_2^{-1} [a_s x+a_l]\,F(dx)\right)^2}{\int
\varphi^2 \,dF} \;\Big|\qquad 0\stackrel{[F]}{\not=} \varphi\in{\cal
D}_k\Big\}
\end{equation}
\end{Exa}
{\small To keep the order of the examples as in section~\ref{exasec}, we place a remark here,
 concerning Example~\ref{corex}}
\begin{Rem}\rm\small \rm
The fact that we are dealing with a one dimensional parameter seems to indicate that
 it should be possible to treat the problem using only one dimensional densities. Factorizations~\eqref{f12-1} and \eqref{f12-2}
 seem to point into the same direction, as they seem to suggest that working with
 \begin{equation} \label{F12-2}
\sigma_1 \sqrt{1-\theta^2}\, P_{\theta}(dx_1,dx_2)=f_{1|2}(\frac{x_1/\sigma_1-\theta x_2/\sigma_2}{\sqrt{1-\theta^2}},\frac{x_2}{\sigma_2})
 \lambda(dx_1)\,F_2(\sigma_2^{-1}dx_2)
\end{equation}
 instead of \eqref{F12-1}, we could allow for any second  marginal $F_2$---possibly even $F_2 \perp \lambda$---and just focus
 on the conditional densities for each fixed $x_2$ section.\\
 Theorem~\ref{dasTHM}, however, excludes that possibility for finite Fisher information. To be fair, one has to admit that anyway, not every
 $F$ with $P_{\theta}=\tau_{\theta} F$ could be allowed for \eqref{F12-2}, but only exactly those achieving this  representation. But even then it is of rather
 marginal interest, as may be seen in the following example:\\
 Consider $Y_1\sim {\cal N}(0,1)$, $Y_2\sim \pm 1$ with $P(Y_2=1)=P(Y_2=-1)=1/2$, $Y_1$, $Y_2$ independent and $F:={\cal L}(Y_1,Y_2)$.
 Then for any $\theta\in\Theta_5$, $X=\tau_{\theta}(Y)=((1-\theta^2)^{\frac{1}{2}}Y_1+\theta Y_2, Y_2),$ and recovering $\theta$ from
 observations of $X$ amounts to estimating $\Ew[X_1|X_2=x_2]$ for $x_2=\pm 1$---a task falling into the usual $\LO_P(n^{-\frac{1}{2}})$-type of
 statistical decision problems;
 if on the other hand, we take $F={\cal L}(Y_1,(1-\alpha^2)^{\frac{1}{2}}Y_1+\alpha Y_2,)$, for any $0<|\alpha|<1$, then, for
 $\theta\not=-(2-\alpha^2)^{-\frac{1}{2}}$, ${\cal L}(X)$  is concentrated on two
 %two parallel
 lines $X_2=a_i+\beta X_1$, $i=1,2$ with
 $\beta=(1-\alpha^2)^{\frac{1}{2}}/[(1-\theta^2)^{\frac{1}{2}}+\theta(1-\alpha^2)^{\frac{1}{2}}]$. But as we assume $F$ to be known, knowledge of $\beta$ is just
 as good as knowledge of $\theta$. Having fixed an observation $X^{(0)}$, $\beta$ may be recovered exactly, as soon as we have found two further
 observations $X^{(1)}$ and $X^{(2)}$ both lying on the same line as $X^{(0)}$, which will happen in finite time almost surely. Thus here a single observation
 must have infinite information on $\theta$---which is just according to our theorem.
\end{Rem}
\section{Consequences for the LAN Approach} \label{Cons}
In general finiteness of Fisher information does not imply $L_2$-differentiability
without additional assumptions like, e.g.\ that for $\lambda^k$ almost all $x$ and for all $\rho \in \R^p$ the map $s\mapsto p_{\theta+s\rho}(x)$
is a.c.\ and the Fisher information ${\cal I}_{\theta}$ is continuous in $\theta$---c.f.\ \citet[17.3~Prop.4]{LC:86}.\\
All examples from section~\ref{exasec}---except for the correlation example, Example~\ref{corex}---%
provide more structure, though. They may all be summarized in the (multivariate) location
scale model of Example~\ref{kls}.
First of all, due to the invariance/dilation relations of Lebesgue measure w.r.t.\ affine transformations,
we may limit attention to the reference parameter ($0,\EM_k$).
Even more though, we have the following generalization of Lemmas by \citet{Ha:1972} (one-dimensional location)
and \citet[Ch.2, Sec.3]{Sw:80} to the multivariate location case
\begin{Prop} \label{klsl2}
Assume that in the multivariate location and scale model~\ref{kls}, Fisher information as
defined in \eqref{Ikls} is finite for some parameter value. Then the model is $L_2$-differentiable
for any parameter value.
\end{Prop}

Hence  Theorem~\ref{dasTHM} gives  a sufficient
condition for these  models to be $L_2$-differentiable and as a consequence to be LAN.\smallskip
\par
On the other hand, $L_2$-differentiability requires finiteness of ${\cal I}_{\theta}$, so that in the multivariate
location and scale case, for all central distributions $F$, the model with central distribution $F$ is $L_2$
differentiable iff $\sup_a {\cal  I}_{\theta}(F;a)<\infty$.\\
In the i.i.d.\ setup \citet[17.3~Prop.2]{LC:86} even show that $L_2$-differentiability is both necessary and sufficient to
get an LAN expansion of the likelihoods in form
\begin{equation}\label{LANiid}
\log dP^n_{\theta+h\!/\!\sqrt{n}}/dP^n_{\theta}=\frac{1}{\sqrt{n}}\sum_{i=1}^n h^{\tau}\Lambda_{\theta}(x_i)-\frac{1}{2}
h^{\tau}{\cal I}_{\theta}h+\Lo_{P_{\theta}^n}(n^0)
\end{equation}
with some $\Lambda_{\theta}\in L_2(P_{\theta})$ and $0\prec{\cal I}_{\theta}=\Ew[\Lambda_{\theta}\Lambda_{\theta}^{\tau}]\prec \infty$, so again in the setup
of the (multivariate) location scale model of Example~\ref{kls} finiteness of Fisher information is both necessary and sufficient to
such an LAN expansion.Altogether we have
\begin{Prop}\label{equivprop}
In models~\ref{1dimlok}, \ref{multilok}, \ref{onescale}, \ref{onels}, \ref{multscal}, \ref{kls},
the following statements are equivalent
\begin{enumerate}
\item[(i)] The respective Fisher information from
        \eqref{Idefkk} is finite for any parameter value.
\item[(ii)] Conditions (ii) of Theorem~\ref{dasTHM} hold for any parameter value.
\item[(iii)] The model is $L_2$-differentiable for any parameter value.
\item[(iv)] The model admits the LAN property \eqref{LANiid} for any parameter value.
\end{enumerate}
\end{Prop}
\begin{Rem}\rm\small
The proof uses the translation invariance and the transformation
property under dilations of $k$-dimensional Lebesgue measure, so there is
not much room for extensions beyond group models induced by subgroups of 
the general affine group.
\end{Rem}

\section{Minimization of the Fisher information} \label{MinI}
Representations~\eqref{Idefk1} resp.\
\eqref{Idefkk} for Fisher information allow for minimization, resp.\ to
maximization of the trace or $\maxev$ of ${\cal I}_{\theta}$ w.r.t.\ the central
distribution $P_{\theta}$ or $F$. In this paper, we settle the questions of (strict) convexity and
lower continuity just as in \citet{Hu:81}, but replace vague topology used in \citet{Hu:81} by weak topology.
This is done in order to establish existence and
uniqueness of a minimizing $F^{(0)}$ in some suitable neighborhood of the (ideal) model.
To this end
define for $a\in\R^p$, $\varphi \in {\cal D}_k$, $\|\varphi\!\circ\!\tau_{\theta} \|_{L_2(F)}\not=0$
\begin{equation}
{\cal I}_{\theta}(F;a;\varphi):=\frac{\big(\int \nabla\varphi^{\tau}D_a+\varphi V_a\, d[\tau_{\theta}F] \big)^2}{\int \varphi^2\, d[\tau_{\theta}F]}
\end{equation}
and
\begin{equation}
\bar {\cal I}_{\theta}(F):=\sup {\cal I}_{\theta}(F;a),\quad {a\in\R^p,\,\; |a|=1}
\end{equation}
\subsection{Weak Lower Semicontinuity and Convexity}
To show weak lower semicontinuity and convexity, we use that for fixed $\varphi\in{\cal D}_k$, $\varphi\not=0\;[P_\theta]$,
$F\mapsto {\cal I}_{\theta}(F;a;\varphi)$ is weak continuous (by definition) and convex (by \citet[Lemma~4.4]{Hu:81}).
Essentially we may then use that the supremum of continuous functions is lower semicontinous and the supremum
of convex functions remains convex; but some  subtle additional arguments are needed
as the set of $\varphi$'s over which we are maximizing may vary from $F$ to $F$; these can be found in
\citet[Proof to Prop.~2.1]{Ru:Ri:10}. Altogether we have shown

\begin{Prop}\label{vlc}
For each $a\in\R^p$, the mapping $F\mapsto{\cal I}_{\theta}(F;a)$ is weakly lower-semicontinuous and convex in $F\in{\cal M}_{1}(\Bq^k)$.
The same goes for $F\mapsto\bar {\cal I}_{\theta}(F)$.
\end{Prop}

\begin{Rem}\rm\small
Using $\Rq^k$ from Definition~\ref{compactifk}, we work with a compact definition space
right away, which moreover is endowed with a separable metric,
 so any subset of probability measures on $\Bq^k$ is tight, hence by Prokhorov's theorem weakly relatively sequentially compact.
\end{Rem}
%
%Weak lower-semicontinuity then entails
%
\begin{Cor}\label{minatt}
In any  weakly closed set ${\cal F}\subset {\cal M}_{1}(\Bq^k)$,
both $\bar {\cal I}_{\theta}$ and ${\cal I}_{\theta;a}$---for fixed $a$---%
attain their  minimum in some $F_0\in {\cal F}$.
\end{Cor}

%
%\begin{proof}{}
%We use weak topology throughout this proof, hence drop
%the attribute ``weak'' here.
%As a closed subset of the compact
%${\cal M}_{1}(\Bq^k)$, ${\cal F}$ is compact. Lower semicontinuity
%ensures that both the sets
%$$\{F\in {\cal F} \,|\, \bar {\cal I}_{\theta}(F)\leq r\},\qquad
%\{F\in {\cal F} \,|\, {\cal I}_{\theta}(F;a)\leq r\}%$$
%are closed for any $r\in\R$ and as subsets of ${\cal F}$,
%are compact, too. Hence by the finite intersection property of compact sets,
%$$\{F\in {\cal F} \,|\, \bar {\cal I}_{\theta}(F)\leq \inf_F {\cal I}_{\theta}(F)\},\qquad
%\{F\in {\cal F} \,|\, {\cal I}_{\theta}(F;a)\leq \inf_F {\cal I}_{\theta}(F;a)\}
%$$
%are both not empty.
%\end{proof}
%
%\subsection{Convexity}

\subsection{Strict Convexity---Uniqueness of a Minimizer}
We essentially take over the assumptions of \citet{Hu:81}; we fix $\theta\in\Theta$ and consider variations in $F$ of the following form:
For $F_i\in{\cal M}(\Bq^k)$ $i=0,1$ consider
\begin{equation}
\tilde F_t:=(1-t)F_0\!\circ\! \iota_{\theta}+ t F_1\!\circ\! \iota_{\theta}
\end{equation}
We distinguish cases ($a$) and ($\bar {\cal I}$), i.e., of a given one-dimensional projection $a\not=0$,
and the corresponding maximal eigenvalue, respectively.
\begin{Prop} \label{maxunique}
Under assumptions
\begin{enumerate}
\renewcommand{\labelenumi}{(\alph{enumi})}
\item \label{uni1} The set  ${\cal F}$ of admitted central distributions $F$ is convex.
\item \label{uni2} There is a $F_0\in{\cal P}$ minimizing
\begin{itemize}
\item[($a$)] ${\cal I}_{\theta}(F;a)$ along ${\cal F}$ and ${\cal I}_{\theta}(F_0;a)<\infty$.
\item[($\bar {\cal I}$)] $\bar {\cal I}_{\theta}(F)$ along ${\cal F}$ and $\bar {\cal I}_{\theta}(F_0)<\infty$.
\end{itemize}
\item \label{uni3} The set where the Lebesgue--density $\tilde f_0$ of $\tilde F_0$ is strictly positive is convex and contains the support of every
$\tilde F_t$ derived from some $F_1\in{\cal F}$.
\item \label{uni4}
\begin{itemize}
\item[($a$)] $\lambda^k(\{x\,|\,                                  a^{\tau}\partial_{\theta} \iota_{\theta}(x)=0\})=0$
\item[($\bar {\cal I}$)] $\lambda^k(\{x\,|\, \exists a: \,|a|=1\;\mbox{ s.t. } a^{\tau}\partial_{\theta} \iota_{\theta}(x)=0\})=0$
\end{itemize}
\end{enumerate}
the map $F\mapsto {\cal I}_{\theta}(F;a)$ (case ($a$))
resp.\ $F\mapsto \bar {\cal I}_{\theta}(F)$ (case ($\bar {\cal I}$)) is strictly convex, hence there is a unique minimizer of $F_0$.
\end{Prop}
\begin{Rem}\rm\small
Assumption (d) holds for the $k$ dimensional location scale model of Example~\ref{kls}:  For symmetric $a_s$ and the scale part $\theta_s$
it holds that
$a_s^{\tau}\partial_{\theta_s} \iota_{\theta}=a_s\theta_s^{-1}x$.
But for any $a_s$ with $|a_s|=1$, $\dim \ker a_s\theta_s^{-1}\leq k-1$, hence $\lambda^k(\ker a_s\theta_s^{-1})=0$.
\end{Rem}
\subsection{Existence of a Maximizer of \mbox{\protect\boldmath $\tr  {\cal I}^{-1}_{\theta}(F)$\unboldmath}}
\begin{Prop}\label{maxexist}
Let ${\cal F}$ be a weakly closed subset of ${\cal M}_{1}(\Bq^k)$. Assume that for all $0\not=a\in\R^p$,
$\min_{F\in{\cal F}}{\cal I}_{\theta}(F;a)>0$. Then the function $F\mapsto \tr  {\cal I}^{-1}_{\theta}(F)$
is weakly upper-semicontinuous on ${\cal F}$, and consequentially, attains its maximum along ${\cal F}$ in some $F_0\in{\cal F}$.
\end{Prop}
%Taking into account that by the asymptotic Cram\'er Rao bound, in the Asymptotic Convolution Theorem, and
%in the Asymptotic Minimax Theorem, see e.g.\ \citet[Thm.'s~3.2.3, 3.3.8]{Ri:94} or
%\citet[Thm.'s~8.8, 8.11]{VdW:98}, ${\cal I}^{-1}_{\theta}(F)$ is the variance of the respective
%optimal estimator $S^0_n$ we hence have shown the existence of an (asymptotic) saddlepoint $(S^0_n,F^0)$
%on any weakly closed subset ${\cal F}$ of ${\cal M}_{1}(\Bq^k)$ such that
%\begin{equation}
%\max_{F\in{\cal F}}\min_{S_n}\tr {\rm asVar}(S_n,F)=\tr {\rm asVar}(S_n^0,F^0)=\max_{F\in{\cal F}}\tr  {\cal I}^{-1}_{\theta}(F)=
%\min_{S_n}\max_{F\in{\cal F}}\tr {\rm asVar}(S_n,F)
%\end{equation}

%\setcounter{section}{-1}
%
%\section{Some preparations}

\appendix%
\small
\makeatletter
\gdef\theThm{\@Alph\c@section.\arabic{Thm}}
\makeatother
\section{Functional Analysis and Generalized Differentiability}\label{app}
%
%In this section we list some basic facts that are used throughout that paper.
%
\subsection{Dense Functions}
\begin{Prop}\label{propdense}
Let $\mu$ be a $\sigma$-finite measure on $\B^k$. Then the set ${\cal C}_c^{\infty}(\R^k,\R)$ is dense in any $L_p(\mu)$,
$p\in[1,\infty)$. In particular, there is a $c_0\in(0,\infty)$ s.t.\
for any $a<b \in \R$ and any $\delta>0$ there is a $\varphi=\varphi_{a,b,\delta}\in{\cal C}_c^{\infty}(\R,[0,1])$,
with $\varphi\equiv 0$ on $[a-\delta;b+\delta]^c$,
$\varphi\equiv 1$ on $[a+\delta;b-\delta]$
and $|\dot \varphi|\leq c_0 \!/\delta$.
\end{Prop}
\begin{proof}{}
Denseness is a consequence of Lusin's Theorem, compare \citet[Thm.~3.14]{Ru:87}.
To achieve the universal bound $c_0$, we may  use
functions $\hat f(t)=\big(\int_0^t\tilde f(s)\,ds\big)\big/\big(\int_0^1\tilde
f(u)\,du\big)$, for $ f(t)=e^{-1/t}$, $\tilde f(t)=f(t)f(1-t)$.
%\end{equation}
\end{proof}
%\begin{Rem}\rm\small
%$c_0$ may be established as $c_0\doteq 2.6054$, which is actually a numerical approximation due to {\tt R 1.6.1} for
%$\max_t|\hat f'(t)|=e^{-4}\!/\int_0^1 \tilde f(s) \,ds$.
%\end{Rem}
%
%
\subsection{Absolute Continuity}
We recall the following characterization of absolute
continuity [notation a.c.] of functions $F:\R\to\R$ that can be found in \citet[Ch.~8]{Ru:87}.
\begin{Thm}\label{aceq}
For  $F\colon[a,b]\to\R$, $a\!<\!b\in\R$ the following statements 1.\ to
3.\ are equivalent
\begin{enumerate}
\item $F$ is a.c. on $[a,b]$
\item
\begin{enumerate}
\item $F'(x)$ exists $\lambda(dx)$ a.e. on $[a,b]$ and $F'\in L_1(\lambda_{|_{[a,b]}})$.
\item $F(x)-F(a)=\int_a^x F'(s) \lambda(ds)$ for all $x\in[a,b]$.
\end{enumerate}
\item There is some $u\in L_1(\lambda_{|_{(a,b)}})$ s.t.\ for $x\in[a,b]$, $F(x)$ has the representation
$$F(x)=F(a)+\int_a^x
u(s)\,\lambda(ds)$$
\end{enumerate}
\end{Thm}
We also recall that a.c.\ functions, are closed under products
\citep[7.2~Prob.4]{Du:02}. In particular, integration by parts is
available. 
%
%\begin{Rem}\rm\small
In this paper, we call a function $F:\R\to\R$ a.c.\ if the equivalent statements 1.\ to 3.\ from Theorem~\ref{aceq}
are valid for each compact interval $[a,b]\subset \R$.
%\end{Rem}
%
%\begin{Lem}\label{acprod}
%Let $I=[a,b]$ and $f \in L_1(\lambda_{|_{[a,b]}})$,  $g \in {\cal C}^1(\R,\R)$, $g \not =0$ on $I$
%such that $fg$ is a.c.\ on $I$.  Then $f$ is a.c.\ on $I$.
%\end{Lem}
%
%
\subsection{Absolute Continuity in Higher Dimensions}
A little care has to be taken about null sets when transferring absolute continuity to higher dimensions.
The next definition is drawn from \citet{Sim:01}.
\begin{Def}\label{ACkdim}
A function $f:(\R^k,\B^k)\to(\R,\B)$ is called {\em absolutely continuous (in $k$ dimensions)\/}, if for every $i=1,\ldots,k$,
there is a set $N_i\in\B^{k-1}$ with $\lambda^{k-1}(N_i)=0$ s.t.\ for $y\in N_i^c$, the function
$f_{i,y}: (\R,\B)\to(\R,\B)$, $x\mapsto f_{i,y}(x)=f((x\!:\!y)_i)$ is a.c.\ in the usual sense.
\end{Def}
%
%\begin{Rem}\rm\small
%The following tempting formulation fails:\\
% $f$ is a.c.\ in $k$ dimensions if it is a.c.\ in the usual sense ``along the line $(x+\mu e_i)_{\mu\in\sR}$.
% for $\lambda^k$ a.e.\ $x$''---just take $k=2$ and the $\lambda^2$ null set $\{(\mu;\mu)_{\mu\in\sR}\}$. Then any
% line lies in the exception set and the  statement is void.
%\end{Rem}
%
In the proof of (ii) $\Rightarrow$ (i) in Theorem~\ref{dasTHM}, we need the following lemma:
\begin{Lem}
\label{Hajekl}
Let $f:\R^k\to\R$ a.c.\ in $k$ dimensions.
Then for each $i=1\ldots,k$
\begin{equation}
\lambda^k(\{f=0\},\{\partial_{x_i}f\not=0\})=0
\end{equation}
\end{Lem}
\begin{proof}{}
Let $g(z)=\Jc_{\{f=0\}\cap\{\partial_{x_i}f\not=0\}}(z)$. Then $g\geq 0$ and Tonelli applies,
so the section-wise defined function $h_y(x):=g((x\!:\!y)_i)$ is measurable for each $y\in\B^{k-1}$ and
defining the possibly infinite integrals $H(y):=\int h_y(x) \, \lambda(dx)$ we get
$
\lambda^k(\{f=0\},\{\partial_{x_i}f\not=0\})=\int g \,d\lambda^k=\int H(y) \lambda^{k-1}(dy)
$. 
But for each $y$ the instances $x$ where $h_y(x)=0$, $h_y'(x)\not=0$ are separated by 
open one-dim.\ sets where $h_y\not=0$, as
$h_y(x)=0$, $h_y'(x)\not= 0$ implies that for some $0<|x'-x|<\ve$, $|f(x')|> |x'-x|\,|h_y'(x)|/2>0$.
Hence at most there can be a countable number of such $x$, and thus $H(y)=0$ for each $y$.
\end{proof}

\subsection{Weak Differentiability}
For proving absolute continuity in Theorem~\ref{dasTHM} we have worked with the notion of weak differentiability;
to this end we compile the following  definitions and propositions again drawn from \citet{Sim:01}, which we have specialized to
differentiation of order one. %We have just changed
%notation and specialized the assertion to hold for the whole $\R^k$ instead of an
%open subset $\Omega$ as integration domain.

\begin{Def}\label{weakderdef}
Let $u\in L_{1,{\ssr loc}}(\lambda^k)$, $1\leq i\leq k$. Then $v_{i} \in L_{1,{\ssr loc}}(\lambda^k)$ is
called {\em weak derivative of $u$\/} (with respect to $x_i$), denoted by $\tilde\partial_{x_{i}} u$, if
\begin{equation}
\int_{\sR^k} u \,\partial_{x_{i}} \varphi\,\, d\lambda^k=-\int_{\sR^k} v_{i}  \varphi\, d\lambda^k\qquad \forall \varphi \in {\cal C}_c^{\infty}(\R^k,\R)
\end{equation}
\end{Def}

\begin{Rem}\rm\small
\begin{ABC}
\item The weak derivative is unique, as for the difference $d=v_i-v_i'$ of two potential candidates,
we have $\int_{\sR^k} d  \varphi\, d\lambda^k=0$ for all $\varphi\in {\cal C}_c^{\infty}(\R^k,\R)$, so by Proposition~\ref{propdense},
$d$ must be $0$ $[\lambda^k]$.
\item Weak derivatives belonging to $L_2(\lambda^k)$ give rise to the space ${\cal W}_{2;1}={\cal W}_{2;1}(\lambda^k)$
of all functions $f:\R^k\to\R$ with
weak derivatives in $L_2(\lambda^k)$ of order one endowed with the norm
%\begin{equation}\label{normdef}
$
\|f\|^2_{{\cal W}_{2;1}}:=\sum_{i= l}^k \,\|\partial_{x_{i}} f\|^2_{L_2(\lambda^k)}
$ %\end{equation}
which is called {\em Sobolev space} of order $2$ and $1$ for which there is a rich theory.
\item The following two propositions---under the additional requirement that $\nabla f$ resp.\ $\tilde \nabla f$
be in $L_2^k(\lambda^k)$, however---may also be found in  \citet[Thm.'s~1 and~2]{Maz:85}.
\end{ABC}
\end{Rem}

%Our proves are adopted from an unpublished manuscript---\citet{Sim:01}. 
%We have just changed
%notation and specialized the assertion to hold for the whole $\R^k$ instead of an
%open subset $\Omega$ as integration domain.
%
%\begin{Def}\label{Sobolevdef}
%The space ${\cal W}_{2;1}={\cal W}_{2;1}(\lambda^k)$ of all functions $f:\R^k\to\R$ with
%weak derivatives in $L_2(\lambda^k)$ of order one endowed with the norm
%\begin{equation}\label{normdef}
%\|f\|^2_{{\cal W}_{2;1}}:=\sum_{i= l}^k \,\|\partial_{x_{i}} f\|^2_{L_2(\lambda^k)}
%\end{equation}
%is called {\em Sobolev space} of order $2$ and $1$.
%\end{Def}
%
\begin{Prop}\label{mazja}
Let $f \in L_{1,{\ssr loc}}(\lambda^k)$ with a weak gradient
$\tilde\nabla f$. Then there is some $\tilde f$, a.c.\ in $k$ dimensions
with usual gradient $\nabla \tilde f$,
such that---up to a $\lambda^k$-null set---$\tilde f=f$  and
 $\tilde \nabla f=\nabla \tilde f$.
\end{Prop}
\begin{proof}{}
%for the case $f\in {\cal W}_{2;1}$
%\citet{Maz:85}\index{Maz'ja}, Theorem~1; a proof of the general case is contained in \citet{Sim:01} and available upon request.
Let again $\Omega_m=[-m,m]$ and consider
$\chi_m\in{\cal D}_k$ with $0\leq \chi_m\leq 1$, $\varphi_m\equiv 0$ on $\Omega_{m+1}^c$, $\chi_m\equiv 1$ on $\Omega_{m}$,
and let $f_m=f\chi_m$. Then $f_m\in L_1(\lambda^k)$ and we have for any $\phi \in {\cal D}_k$
$$
-\int f_m \partial_{x_i}\phi\, d\lambda^k= -\int \chi_m \phi \,\tilde \partial_{x_i} f\, d\lambda^k=\int f \partial_{x_i}(\chi_m \phi)\, d\lambda^k=
\int f (\phi \,\partial_{x_i}\chi_m + \chi_m \partial_{x_i}\phi)\, d\lambda^k
$$
so that $f_m$ is weakly differentiable and
$
\tilde \partial_{x_i} f_m= \chi_m  \tilde \partial_{x_i} f+f \partial_{x_i} \chi_m  \in L_1(\lambda^k)
$.
By Fubini we obtain some $N_{m,i}\in\B^{k-1}$ with $\lambda^{k-1}(N_{m,i})=0$ such that $v_m:\R^{k-1}\to\R$ defined as
$$
v_m(y):=%\left\{
%\begin{array}
%{ll}
\int_{\sR} |\tilde \partial_{x_i} f_m((t\!:\!y)_i)|\,\lambda(dt)\;\; %&
\mbox{ for }y\in N_{m,i}^c\;\;%\\
\mbox{ and }\;\;0%&
\mbox{ else}
%\end{array}
%\right.
$$
is finite for $y\in \R^{k-1}$, lies in $L_1(\lambda^{k-1})$ and $\int_{\sR^{k-1}} v_m\,d\lambda^{k-1}= \|\tilde \partial_{x_i}f_m\|_{L_1(\lambda^k)}.$
Thus we may define to $x\in \R$
\begin{equation}
F_m((x\!:\!y)_i):=
%\left\{
%\begin{array}
%{ll}
\int^{x}_{-\infty} \tilde \partial_{x_i} f_m((t\!:\!y)_i)\,\lambda(dt)\;\;% &
\mbox{ for }y\in N_{m,i}^c\;\;%\\
\mbox{ and }\;\;0%&
\mbox{ else}
%\end{array}
%\right.
\end{equation}
Apparently, $F_m\in L_{1,{\rm\SSs loc}}(\lambda^k)$ and for $y\in N_{m,i}^c$,  $x\mapsto F_m((x\!:\!y)_i)$ is a.c.
Let $\phi \in{\cal D}_k$; then Fubini yields
\begin{eqnarray*}
I&:=&\int_{\sR^k} \phi F_m d\lambda^k=%
%&=&\int_{N_{m,i}^c}\int_{\sR} \phi((x\!:\!y)_i) \int_{\sR} \Jc_{\{t\leq x\}} \tilde \partial_{x_i}f_m((t\!:\!y)_i) \,\lambda(dt)\,\lambda(dx)\,\lambda^{k-1}(dy)=\\
%&=&
\int_{N_{m,i}^c}\int_{\sR} \tilde \partial_{x_i}f_m((t\!:\!y)_i)\int_{\sR} \Jc_{\{x\geq t\}}  \phi((x\!:\!y)_i) \,\lambda(dx)\,\lambda(dt)\,\lambda^{k-1}(dy)%=\\
\end{eqnarray*}
So far we do not know if the inner integral on the RHS is in $L_1(\lambda^k)$, so another localization argument is needed. To this end let
$\psi \in{\cal D}_k$, $\psi\equiv 1$ on $\Omega_{m+1}$, $\psi\equiv 0$ on $\Omega_{m+2}^c$; then as $f_m, \tilde\partial_{x_i} f_m\equiv 0$ on $\Omega_{m+1}^c$, we have
$f_m\equiv f_m\psi$, $\tilde\partial_{x_i} f_m\equiv \psi \tilde\partial_{x_i} f_m$, and $f_m \partial_{x_i} \psi \equiv 0$. For with $u=(t\!:\!y)_i$ define the function
$$\varphi(u):=\psi(u)\int  \Jc_{\{x\geq t\}}  \phi((x\!:\!y)_i) \,\lambda(dx),$$
which clearly lies in ${\cal D}_k$. Fubini and the definition of weak differentiability entail
\begin{eqnarray*}
I%&=&\int_{N_{m,i}^c}\int_{\sR} \tilde \partial_{x_i}f_m((t\!:\!y)_i)\psi((t\!:\!y)_i)\int_{\sR} \Jc_{\{x\geq t\}}  \phi((x\!:\!y)_i) \,\lambda(dx)\,\lambda(dt)\,\lambda^{k-1}(dy)=\\
&=&\int_{N_{m,i}^c}\int_{\sR} \tilde \partial_{x_i}f_m((t\!:\!y)_i)\varphi((t\!:\!y)_i)\,\lambda(dt)\,\lambda^{k-1}(dy)=%=\\
%&=&
\int_{\sR^k}\varphi \tilde \partial_{x_i}f_m\,d\lambda^{k}=-\int_{\sR^k} f_m  \partial_{x_i}\varphi\,d\lambda^{k}
\end{eqnarray*}
But, 
$\partial_{x_i}\varphi(u)=\partial_{x_i}\psi(u) \int_{\sR} \Jc_{\{x\geq t\}}(x)  \phi((x\!:\!y)_i) \,\lambda(dx) - \psi(u)\phi(u) $,
as $f_m \partial_{x_i} \psi \equiv 0$, $f_m\equiv f_m\psi$ we get
\begin{equation}
I=\int_{\sR^k} \phi F_m d\lambda^k=\int_{\sR^k} f_m \psi \phi \,d\lambda^{k}=\int_{\sR^k} f_m \phi \,d\lambda^{k},
\end{equation}
Because $\phi$ was arbitrary in ${\cal D}_k$, $F_m=f_m\;\;[\lambda^k]$, and by letting $m\to\infty$ we may extend this to $\R^k$.
Fubini then provides a $\lambda^{k-1}$-null set $S_i$ s.t.\ for $y \in S_i^c$ the projection set
$S_i^{(y)}:=\{x \in \R\,:\; (x\!:\!y)_i\in S_i\}$ has $\lambda$-measure $0$.
Let $N_i:=\bigcup_m N_{m,i}$; then $\lambda^{k-1}(N_i)=0$, and for $y\in N_i^c$ the functions $x\mapsto F_m((x\!:\!y)_i)$ are a.c.,
hence continuous in particular. For $y\in (N_i\cup S_i)^c$, $x\in (S_i^{(y)})^c$ and $k\in\N$, even
$
F_m((x\!:\!y)_i)=F_{m+k}((x\!:\!y)_i)
$,
and hence by continuity, for all $y\in (N_i\cup S_i)^c$, $F_m((x\!:\!y)_i)=F_{m+1}((x\!:\!y)_i)$ for all $x$.
Hence, writing again $u=(t\!:\!y)_i$, this gives a unique function $\tilde f_i\in L_{1,{\rm\SSs loc}}(\lambda^k)$
defined as
\begin{equation}
\tilde f_i(u):=\left\{
\begin{array}
{ll}
\lim_m F_m(u)&\mbox{for }u\in\R^k,\quad y\in (N_i\cup S_i)^c\\
0&\mbox{else}
\end{array}
\right.
\end{equation}
s.t.\ that $\tilde f_i$ is a.c.\ w.r.t.\ $x_i$ in the sense that there is $\lambda^{k-1}$-null set $\tilde N_i$ s.t.\ for
$y\in \tilde N_i^c$ the function $x\mapsto \tilde f_i((x\!:\!y)_i)$ is a.c. By construction, 
$\lambda^k(\{\tilde f_i\not=f\}\cup\{\tilde \partial f\not=\partial \tilde f_i\})=0$. 
Applying this argument for each $i=1,\ldots,k$, we see that there is a function $\tilde f$ which is a.c.\ in $k$ dimensions,
s.t.\ $\lambda^k(\{\tilde f\not=f\}\cup\{\tilde \nabla f\not=\nabla \tilde f\})=0$.
\end{proof}
\begin{Prop}\label{mazja2}
Let $f\in L_{1,{\ssr loc}}(\lambda^k)$ be a.c.\ in $k$ dimensions.
If its classical  partial derivatives $\partial_{x_i}f $,
 are extended by $0$ on those lines where absolute continuity fails,  and the so extended
  gradient belongs to $L^k_{1,{\ssr loc}}(\lambda^k)$, then there is a weak  gradient of
  $f$ and the extended gradient can be taken as a version of  the weak gradient.
\end{Prop}
\begin{proof}{}
As $f$ is a.c.\ in $k$ dimensions there exist $N_i\in \B^{k-1}$ such that for $y \in N_i^c$ the functions
$x\mapsto f((x\!:\!y)_i)$ are a.c. Let $\phi\in {\cal D}_k$ and $y \in N_i^c$. Then $x\mapsto \phi((x\!:\!y)_i)\in{\cal D}_1$ and thus
by integration by parts, for $y \in N_i^c$, we have
\begin{equation}
\int_{\sR} f((x\!:\!y)_i)\, \partial_{x_i}\phi((x\!:\!y)_i)\,\lambda(dx)=-
\int_{\sR} \phi((x\!:\!y)_i) \, \partial_{x_i}f((x\!:\!y)_i)\,\lambda(dx) \label{hic}
\end{equation}
Obviously, extending $f \, \partial_{x_i}\phi$, $\phi\,\partial_{x_i}f $ by $0$ on $y\in N_i$, these two functions belong to $L_1(\lambda^k)$.
Fubini thus yields a set $\tilde N_i\in\B^{k-1}$, $\lambda^{k-1}(\tilde N_i)=0$, s.t.\ for $y\in \tilde N_i^c$,
$x\mapsto [f \, \partial_{x_i}\phi]((x\!:\!y)_i)$, $x\mapsto [\phi\,\partial_{x_i}f]((x\!:\!y)_i)$ belong to $L_1(\lambda)$.
Hence by Fubini
%\begin{eqnarray*}
$\int_{\sR^k} f \, \partial_{x_i}\phi\, d\lambda^k%&=&\int_{(N_i\cup\tilde N_i)^c} [ \int_{\sR} f((x\!:\!y)_i)\, \partial_{x_i}\phi((x\!:\!y)_i)\,\lambda(dx)]\,\lambda^{k-1}(dy)=\\
%&\stackrel{{\ssr \eqref{hic}}}{=}&-\int_{(N_i\cup\tilde N_i)^c} [ \int_{\sR} \phi((x\!:\!y)_i)\, \partial_{x_i}f((x\!:\!y)_i)\,\lambda(dx)]\,\lambda^{k-1}(dy)=%\\
%&=&
=-\int_{\sR^k} \phi \, \partial_{x_i}f\, d\lambda^k$. 
%\end{eqnarray*}
As $\partial_{x_i}f\in L_{1,{\rm\SSs loc}}(\lambda^k)$ by definition of absolute continuity in $k$ dimensions,
this possibly extended $\partial_{x_i}f$ is a weak derivative of $f$.
%for the case that the gradient already belongs to $L_2(\lambda^k)$, which entails that $f\in L_{1,{\ssr loc}}(\lambda^k)$,
%\citet{Maz:85}\index{Maz'ja}, Theorem~2; a proof of the general case is contained in \citet{Sim:01} and available upon request.
\end{proof}
\begin{Rem}\rm\small
Having this ``almost'' coinciding of weak differentiability and absolute continuity in $k$ dimensions in mind,
we drop the notational difference of weak and classical derivatives.
\end{Rem}
\section{Proofs}\label{proofs}
\subsection{Preparations}
Before proving Theorem~\ref{dasTHM}, some preparations are
needed.

%\begin{Rem}\rm\small
We want to parallel the proof given in \cite{Hu:81} credited to T.~Liggett: The idea is to define for given $a\in\R^p$
linear functionals $\tilde T_{a;i}$ on the dense subset ${\cal C}^{\infty}(\Rq^k,\R)$ of $L_2(P_{\theta})$ as
\begin{equation}\label{Tidef}
\tilde T_{a;i}:{\cal C}^{\infty}(\Rq^k,\R)\to\R,\qquad \tilde T_{a;i}(\varphi):=\int D_{a;i}\,\pxI{i}\varphi \,dP_{\theta}.
\end{equation}

\begin{Rem}\label{welldef?}\rm\small
As also true for the one-dimensional location model treated in \citet{Hu:81} and
in the one-dimensional scale model in \citet{Ru:Ri:10}, 
it is not clear  \`a priori whether this is a sound
definition, i.e., whether $\tilde T_{a;i}$ respect equivalence classes of functions in $L_2(P_{\theta})$:\\
%First of all we only know that ${\cal I}(F;a)$ being finite, $D_{a;i}\,\pxI{i}\varphi+V_a\varphi\in L_1(P_{\theta})$. But
% as $V_a\in L_2(P_{\theta})$,  integrability of $D_{a;i}\,\pxI{i}\varphi$ follows.\\
As by (Dk) resp.\ (D1), $D_{a;i}$ is continuously differentiable, 
it is bounded on compacts, hence $D_{a;i}\,\pxI{i}\varphi \in  L_1(P_{\theta})$
for any $\varphi \in{\cal C}^{\infty}(\Rq^k,\R)$;  but
even then, it is still not clear whether \eqref{Tidef} makes a definition:
Take $x^{(0)}\in\R^k$ so that $x^{(0)}_{i}D_{a;i}(x^{(0)})\not=0$ and
$P_{\theta}$ Dirac measure for $\{x^{(0)}\}$.\\
Then obviously,
$\varphi_0(x)=(x^{(0)})^{\tau}(x-x^{(0)})=0$ $P_{\theta}(dx)$-a.e., but it also holds,
$\pxI{i}\varphi_0(x)D_{a;i}(x)=
%V_a(x^{(0)})\varphi(x^{(0)})+\pxI{i}\varphi(x^{(0)})D_{a;i}(x^{(0)})=
x^{(0)}_{i}D_{a;i}(x^{(0)})$  $P_{\theta}(dx)$-a.e., with the consequence that, although
$\varphi_0=0\;[P_{\theta}]$, $\tilde T_{a;i}(\varphi_0)\not=\tilde T_{a;i} (0)=0$.
 %, that is
% Definition~\ref{Itdef} does not necessarily respect equivalence classes of functions in $L_2(P_{\theta})$.
Of course,   $\varphi_0$ must be modified away from $x^{(0)}$ to some $\tilde\varphi$ s.t.\ 
$\tilde\varphi$ belongs to  ${\cal C}^{\infty}(\Rq^k,\R)$.
Luckily enough, this case cannot occur under condition (i) of Theorem~\ref{dasTHM},
as then 
\begin{equation} \label{welldef}
\varphi=0\;[P_{\theta}]\qquad \Longrightarrow \qquad
\tilde T_{a;i}(\varphi)=\int \pxI{i}\varphi\,D_{a;i} \,dP_{\theta}=0
\end{equation}
which may be proved just along the lines of the first 
paragraph of \citet[Proof to Thm.~2.2]{Ru:Ri:10}.
Due to linearity of differentiation, evaluated member-wise in an
$P_{\theta}$-equivalence class, this shows that $\tilde T$
respects $P_{\theta}$-equivalence classes.
\end{Rem}

Next we need a lemma showing denseness of  certain sets in suitable
$L_2$'s. To do so we define for $i,j=1,\ldots ,k$% and $a=a_j$
\begin{equation}
{\cal D}_{D;i,j}:=\{D_{a;j}\,\partial_{x_{i}}\phi\,|\,\phi\in{\cal D}_{k},\;a\in\R^p\}
\end{equation}
and recalling that $K_j:=\{e_j^{\tau}D=0\}$ we introduce the decompositions corresponding to \eqref{p0def}
\begin{equation}
P_{\theta}:=P^{(j)}_{\theta}+\bar P^{(j)}_{\theta},\qquad \bar P^{(j)}_{\theta}(\cdot):=P_{\theta}(\cdot\cap K_j).\label{p0def}
\end{equation}

\begin{Lem}\label{densel}
${\cal D}_{D;i,j}$ is dense in $L_{2}(Q)$ for any $\sigma$-finite measure on $\B^k\cap K_{j}^c$. In particular it is
measure-determining for $\B^k\cap K_{j}^c$.
\end{Lem}
\begin{proof}{}
Approximating $f \in L_{2}(Q)$ in $L_{2}(Q)$ by $f_n:=f\Jc_{\Omega_n}$ with
$\Omega_n=[-n,n]^k$,
we may restrict ourselves to $\Omega_N$ for $N$ sufficiently large.  Thus we
have to show that for each interval $J\subset K_j^c$ and each $\ve>0$, there is a $\hat \varphi$ in ${\cal D}_{D;i,j}$ with
$\|\Jc_J-\hat \varphi\|_{L_2(P_{\theta})}\leq \ve$. \\
To this end fix $a\in\R^p$; as $D_{a;j}$ is continuous,
the set $K_j^c$ is open, hence is the countable union
of $k$ dimensional intervals $J_m:=(l^{(m)};r^{(m)})$, $m\in \N$, with $l^{(m)}<r^{(m)}$ and $|D_{a;j}|>0$ on $J_m$. \\
So it suffices to show that any indicator to an interval $I=[\tilde l;\tilde r]$, $I\subset J_m$ with endpoints
s.t.\  $Q(\partial I)=0$
may be approximated in $L_2(Q)$ by functions in ${\cal D}_{D;i,j}$. 
But, for given $\ve>0$, Proposition~\ref{propdense} provides an element
$\varphi_0 \in {\cal C}_c^{\infty}(\R^k,\R)$ such that $\|\varphi_0-\Jc_{I}\|_{Q}<\ve$. By construction its anti-derivative 
$\psi_0((x\!:\!y)_i):=\int \Jc(_{\tilde l^{(m)}_i}\le z\le x)\varphi_0((z\!:\!y)_i))/D_{a;j}((z\!:\!y)_i))\,\lambda(dz)$ lies in
${\cal C}^{\infty}(\Rq^k,\R)$ hence $\varphi_0$ in ${\cal D}_{D;i,j}$. In particular we may
approximate the $Q$ measure for $k$-dimensional intervals disjoint to $K_j$, which determines $Q$.
\end{proof}

\subsection{Proof of the Main Theorem}
\noindent\textbf{\boldmath$\mbox{\rm\bf(ii)} \Rightarrow \mbox{\rm\bf(i)}$ of Theorem~\ref{dasTHM}}\\
In order to avoid specializing the case $k=1$, define $V:=0$ there.\\
Fix $a\in\R^p$, $|a|=1$. $f$ being a density, $\lambda(\{f=0,\pxI{i} f\not=0\})=0$ for each $i$ and we may write
\begin{eqnarray*}
\lefteqn{\int  (\pxI{i}\!\varphi)  D_{a;i}\,dP_{\theta}\!\stackrel{}{=}\!\int_{K^c}\!\!  (\pxI{i}\!\varphi)  D_{a;i}\,dP_{\theta}\!=\!
\int_{\iota(K^c)}\!\!\!\!  [(\pxI{i}\!\varphi)  D_{a;i}]\!\circ\!\tth1\,dF%=%}&&\\
%&
\!\!\stackrel{\Ss(ii)(a)}{=}\!\!\! %&
\int_{\iota(K^c)} \!\!\!\! [(\pxI{i}\!\varphi) D_{a;i}]\!\circ\!\tth1 f\,d\lambda^k=}&&\\
\!\!\!\!\!\!\!&\!\!\!=&\!\!\!\!\!\!\!\int_{K^c}  (\pxI{i}\!\varphi) D_{a;i} f_{\theta}\,\dttx \,d\lambda^k
\stackrel{(\ast)}{=} -\int_{K^c} \!\!  \varphi\, f_{\theta}\, \pxI{i}(D_{a;i}\,\dttx ) +
\varphi \, D_{a;i}\,\dttx\,  \pxI{i}f_{\theta}
\,d\lambda^k=\\
\!\!\!\!\!\!\!&\!\!\!=&\!\!\!\!\!\!\!-\int_{K^c}\!\!\!   \varphi p_{\theta}\, [\frac{\pxI{i}(\dttx\,D_{a;i})}{\dttx} +
\frac{D_{a;i} \pxI{i}f_{\theta}}{f_{\theta}}]
\,d\lambda^k%=\\
%&
\!\stackrel{(\ast\ast)}{=}\!%&
-\int \!\!  \varphi p_{\theta}\, [\frac{\pxI{i}(\dttx\,D_{a;i})}{\dttx} +
\frac{D_{a;i} \pxI{i}f_{\theta}}{f_{\theta}}]
\,d\lambda^k
\end{eqnarray*}
In equation $(\!\,\ast\!\,)$ we use that by (ii)(c), on $K^c$, $f$ is a.c.\ $\lambda^k$ a.e.\ so that integration by parts
integration by parts is available without having to care about border values due to (ii)(b). By (ii)(d) the resulting integrand on the RHS of
$(\,\!\ast\,\!)$ is in $L_2(P_{\theta})$.
In equation $(\!\,\ast\ast\,\!)$, we used the fact that in each expression considered above, there appears at least one $D_{a;i}$
or a derivative $\partial_{x_i}D_{a;i}$;
Lemma~\ref{Hajekl} applies and hence
$$\lambda^{k}(\{D_{a;i}=0\},\{\partial_{x_i}D_{a;i}\not=0\})=0$$

\noindent\textbf{\mbox{Representations~\eqref{ass1} and \eqref{ass2}:}}\\
 Writing out
$
\partial_xf_{\theta}=(\partial_x \iota_{\theta})(\partial_x f)\!\circ\! \iota_{\theta}
%\pxI{i}f_{\theta}=\Tsum_{m=1}^k (\pxI{m} f)\!\circ\! \iota_{\theta} \pxi \iota_{\theta;m}
$,
we see that
\begin{eqnarray}
&&D_{a;\cdot}^{\tau} \partial_xf_{\theta} = a^{\tau} (\partial_{\theta} \iota_{\theta}) J (\partial_x \iota_{\theta})(\partial_x f)\!\circ\! \iota_{\theta}%=\\
%&=&
=
a^{\tau} (\partial_{\theta} \iota_{\theta}) (\partial_x f)\!\circ\! \iota_{\theta}=a^{\tau}\partial_\theta f_{\theta}\qquad \label{hinher}
%\Tsum_{i=1}^k D_{a;i} \pxI{i}f_{\theta}&=&%\nonumber\\
%&=&
%\Tsum_{i,j,l,m=1}^k a_j \pthj\tathl  J_{l,i}
% (\pxI{m} f)\!\circ\! \iota_{\theta} \pxi \iota_{\theta;m}=\nonumber\\
%&=&\Tsum_{j,l=1}^k a_j \pthj\tathl  (\pxI{l} f)\!\circ\! \iota_{\theta} =a^{\tau} \partial_\theta f_{\theta}
\end{eqnarray}
Thus we get
%\begin{eqnarray*}
%&&
$$
\frac{
\sum_i\pxI{i}[D_{a;i}\,\dttx ]
%{\rm div}_x[D_{a;\cdot}\,\dttx ]
}{\dttx}+
\frac{ D_{a;\cdot}^{\tau} \partial_xf_{\theta}}{f_{\theta}}-V_a
\stackrel{\mbox{\tiny Def.~$V$}}{=}
%\frac
\frac{a^{\tau}\pth\,\dttx }{\dttx}+
\frac{D_{a;\cdot}^{\tau} \partial_x f_{\theta}}{f_{\theta}}
%{\int \varphi^2 \, dP_{\theta}}
%=\\
%&
\stackrel{\mbox{\tiny\eqref{hinher}}}{=}%&
%\frac
\frac{a^{\tau}\pth\,\dttx }{\dttx}+
\frac{a^{\tau}\pth f_{\theta}}{f_{\theta}}%=\frac{a^{\tau}\pth\,p_{\theta} }{p_{\theta}}
=a^{\tau}\Lambda_{\theta},
$$
%\end{eqnarray*}
so $\Lambda_{\theta}\in L_2^p(P_{\theta})$ by (ii)(d) and hence
% and writing $f_{\theta}=f_{\theta},$ $p_{\theta}=f_{\theta}\,\dttx$,
\begin{eqnarray*}
&&%\frac
{\left(\int  \nabla\varphi^{\tau} D_{a}+\varphi V_a\,dP_{\theta}\right)^2}
%{\int \varphi^2\,dP_{\theta}}
=
{\left(\int \varphi \frac{a^{\tau}\pth\,p_{\theta} }{p_{\theta}} \, dP_{\theta}\right)^2}
%{\int \varphi^2 \, dP_{\theta}}
\leq \int [a^{\tau}\Lambda_{\theta} ]^2 \, dP_{\theta} \,\,{\int \varphi^2 \, dP_{\theta}},
\end{eqnarray*}
which shows that ${\cal I}(F;a)\leq \int (a^{\tau}\Lambda_{\theta})^2\, dP_{\theta}$.
The upper bound may be approximated by a sequence $\varphi_n\in {\cal D}_k$ tending
to $a^{\tau}\Lambda_{\theta}$ in $L_2(P_{\theta})$ entailing \eqref{ass1} and \eqref{ass2}.\\
%\end{f1proof}
%
%
%
%\noindent

\noindent\textbf{\boldmath$\mbox{\rm\bf(i)} \Rightarrow \mbox{\rm\bf(ii)}$ in Theorem~\ref{dasTHM}}\\
%\\
We will give a proof largely paralleling \citet{Hu:81}, although
we may skip some of his arguments.\\[0.3ex]
{\bf Well defined operators and Riesz-Fr\'echet: }%\\[0.3ex]
 We consider the linear functionals $\tilde T_{a;i}$ from \eqref{Tidef}, defined on the
 dense subset ${\cal C}^{\infty}(\Rq^k,\R)$ of $L_2(P_{\theta})$, which are well defined
 due to \eqref{welldef}. In particular $\tilde T_{a;i}$
are  bounded linear operators with squared
operator norms bounded by ${\cal I}_{\theta}(F;a)$, hence can be extended by continuity
to continuous linear operators $T_{a;i}: L_2(P_{\theta}) \to \R$ with the same
operator norms. Thus  Riesz Fr\'echet %[\citet{Ru:68},...]
applies, yielding generating elements $g_{a;i}\in L_2(P_{\theta})$ s.t.
\begin{equation} \label{gidef} T_{a;i}(\varphi)= -\int g_{a;i} \varphi \,dP_{\theta}
\qquad \forall \varphi \in L_2(P_{\theta})\qquad\quad\mbox{and } \|g_{a;i}\|_{L_2(P_{\theta})}^2=\|\tilde T_{a;i}\|
\end{equation}
We conclude inductively for $i=1,\ldots,k$.\\[0.5ex]
\fbox{\mbox{\protect\boldmath ${i=1}$\unboldmath}}
{\bf Using Fubini: }%\\[0.3ex]
We have for $\varphi\in{\cal D}_k$
\begin{equation}
T_{a;1}(\varphi)=\int \,D_{a;1}\, \pxI{1}\varphi\,dP_{\theta}\label{a1gl}
\end{equation}
On the other hand by the fundamental theorem of calculus,
%\begin{eqnarray*}
%\lefteqn{
$$T_{a;1}(\varphi)=-\int \varphi  g_{a;1}\,dP_{\theta}=%}&&\\
%&=&=
\int_{\sR^{k}} \int_{\sR} \Jc_{\{x_1\geq y_1\}}
\pxI{1}\varphi(x_1,y_{2:k}) \lambda(dx_1)\,  g_{a;1}(y)\,P_{\theta}(dy).$$
%\end{eqnarray*}
Now for each compact $A$, the integrand $\tilde h_1(x_1,y;A)=\Jc_A(x_1)\Jc_{\{x_1\geq y_1\}}
 g_{a;1}(y)$ is  in %the product space
$L_{1}(\lambda(dx_1) \otimes P_{\theta}(dy))$.
Fubini for Markov kernels thus yields  a $\lambda \otimes P_{\theta;\,2:k}$-null set $N_1$
such that for $(x_1,y_{2:k})\in N_1^c$, $x_1\in A$, the %section--wise defined
 function $h_{1}(y_1):= \tilde h_1(x_1,y;A)$
belongs to  $L_{1}\big(P_{\theta;\,1|2:k}(dy_1|y_{2:k})\big)$. %A closer look yields $N_1=\R \times \check N_1$
%for some $\check N_1 \in \B^{k-1}$.
We now define for $D_{a;1}\not=0$ the function $p^{(a;1;A)}_{1|2:k}$ as
\begin{equation}\label{p1def}
[D_{a;1}p^{(a;1;A)}_{1|2:k}](x_1,y_{2:k}):=\left\{\begin{array}
{ll} \int h_{1}(y_1)\,P_{\theta;\,1|2:k}(dy_1|y_{2:k}) &%\\
\mbox{for}\quad (x_1,y_{2:k})\in  N_1^c,\;x_1\in A\\
\qquad0&\mbox{else}
\end{array}\right.
\end{equation}
where obviously the dependence in $A$ is such that for another compact $A'\supset A$,
$p^{(a;1;A)}_{1|2:k}=p^{(a;1;A')}_{1|2:k}\Jc_A(x_1)$. Hence for arbitrary $x_1$ take $A$ such that $x_1\in A$
and eliminate the index $A$ in the superscript where it is clear from the context.
We also note that by Cauchy-Schwarz, for $P_{\theta;2:k}(dy_{2:k})$-a.e. $y_{2:k}$,
\begin{equation} \label{xfbounda}
|[D_{a;1}p^{(a;1)}_{1|2:k}](x_1,y_{2:k})|^2 \leq \int g_{a;1}(y)^2 \,P_{\theta;\,1|2:k}(dy_1|y_{2:k}) < \infty
\end{equation}
\\[0.3ex]
%or, equivalently,
%$$
%p^{(a;1)}_{1|2:k}(x_1,y_{2:k})=\Jc_{\{D_{a;1}\not=0\}}\int_{-\infty}^{x_1}  g_{a;1}(y)\,P_{\theta;\,1|2:k}(dy_1|y_{2:k})/D_{a;1}(x_1,y_{2:k}).$$
%
{\bf Getting rid of the dependence on $a$: }%\\[0.3ex]
To understand, how $p^{(a;1)}_{1|2:k}$ is related to $p^{(a';1)}_{1|2:k}$ for $a\not=a' \in \R^p$, we consider
again \eqref{Tidef}, \eqref{p1def}: Both sides of the latter must be of form
$\tilde W^{\tau}a$ for some $\R^p$ valued $\tilde W$ independent of $a$; in particular
\begin{equation}\label{wdef}
g_{a;1}=w_1^{\tau}a
\end{equation}
for some $w_1\in L^p_2(P_{\theta})$.
Hence
\begin{equation}
p^{(a;1)}_{1|2:k}=p^{(a';1)}_{1|2:k}\qquad \mbox{ on }\{D_{a;1}^{(k)}\not=0\}\cap \{D_{a';1}^{(k)}\not=0\}
\end{equation}
and, as we only need $p$ orthogonal values of $a$ to specify $w_1$, we arrive at a
maximally extended  $p^{(1)}_{1|2:k}$ defined on
$K_1^c=\{D_{\cdot;1}\not=0\}$. Also, for $P_{\theta;2:k}(dy_{2:k})$-a.e. $y_{2:k}$,
\begin{equation} \label{xfbound}
|[D_{a;1}p^{(1)}_{1|2:k}](x_1,y_{2:k})|^2 \leq |a|^2\int |w_{1}(y)|^2 \,P_{\theta;\,1|2:k}(dy_1|y_{2:k})
\end{equation}
\\[0.3ex]
{\bf
{\mbox{\protect\boldmath $p^{(1)}_{1|2}$ \unboldmath}}
 \bf is a density: }%\\[0.3ex]
Plugging in this maximal definition, we get for $\varphi \in {\cal D}_k$, using $A={\rm supp}(\varphi)$,
\begin{eqnarray}
T_{a;1}(\varphi)&=&%\!\!&\!\!=\!\!&\!\!
\int \,D_{a;1}\, \pxI{1}\varphi\,dP_{\theta}=
%\int \Jc_{\{D_{a;1}\not=0\}} \,D_{a;1}\, \pxI{1}\varphi\,dP_{\theta}= %\nonumber\\
%\!\!&\!\!=\!\!&\!\!
\int [D_{a;1}\,   \pxI{1}\varphi\,\, p^{(1)}_{1|2:k} ](x_1,y_{2:k})
\lambda(dx_1) P_{\theta;\,2:k}(dy_{2:k})\label{a2gl}.
\end{eqnarray}
where integrability of the integrands follows
from Remark~\ref{D2Rem}(e) and (C1)/(Ck), and for the right one from \eqref{xfbound},
which also entails that $P_{\theta;\,2:k}(dy_{2:k})$-a.s., 
$x_1\mapsto D_{a;1}\,p^{(1)}_{1|2:k}$ is the $\lambda$-density
of a $\sigma$-finite signed measure.
Hence, we have shown that $P_{\theta}(dx_1,dy_{2:k})$
and  $p_{1|2:k}(x_1,y_{2:k}) \lambda(dx_1)P_{\theta;\,2:k}(dy_{2:k})$ when restricted to $K_1^c$ define the same functional on the set
${\cal D}_{D;1,1}$, which
is measure-determining for $\B^k\cap K_1^c$ due to Lemma~\ref{densel}.\\
Therefore, the restriction to compacts $A$ can be dropped entirely, and we may work with $A=\R$.
Using Fubini once again, we see that  on $K_1^c$, there is a  $P_{\theta;\,2:k}(dy_{2:k})$-null set $\tilde N_1$, s.t.\
for fixed $y_{2:k}\in \tilde N_1^c$, the function $p^{(1)}_{1|2:k}(x_1,y_{2:k})$
is a Lebesgue density of the regular conditional distribution $P^{(0)}_{\theta;1|2:k}(dx_1|y_{2:k})$,
hence non negative and in $L_1(\lambda)$.
%On the other hand, $g_{a;1}$ being in $L_2(P_{\theta})$, there is another
%$P_{\theta;\,2:k}(dy_{2:k})$--null-set $\bar N_1$ such that  for $y_{2:k}\in \bar N_1^c$,
%$|\int  g_{a;1}^2 \, dP_{\theta;1|2:k}|<\infty$. Extend  $\bar N_1 \cup \tilde N_1$ to $\tilde M_1:=\R\times(\bar  N_1\cup  \tilde N_1)$,
%which is of $P^{(0)}_{\theta}$-measure 0.
\\[0.3ex]
{\bf Replacing  {\mbox{\protect\boldmath $K_1$ \unboldmath}} by {\mbox{\protect\boldmath $K$ \unboldmath}:} }%\\[0.3ex]
Similarly as for the dependence on $a$, we may extend the definition of $p^{(1)}_1(x_1,y_{2:k})$
to the set $K^c$:
Any $\partial_{x_i} \varphi$ for $\varphi \in {\cal D}_k$ may also be interpreted as
$\partial_{x_j} \tilde \varphi$ for some $\tilde \varphi\in {\cal D}_k$. More specifically, $\tilde \varphi=\varphi \!\circ\! \pi_{i,j}$ with $\pi_{i,j}$
the permutation of
coordinates $i$ and $j$. Thus introducing for $1\leq i,l\leq k$
operators $\tilde T_{a;i,j}:{\cal D}_k\to \R$, $\varphi \mapsto \int \partial_{x_i} \varphi D_{a;j} dP_{\theta}$,
we amply see their boundedness in operator norm by $\|T_{a;j}\|$, hence extending them to $L_2(P_{\theta})$ as before, giving operators $T_{a;i,j}$,
we also get generating elements $g_{a;i,j}\in L_2(P_{\theta})$ by Riesz-Fr\'echet and eventually, using denseness
of ${\cal D}_{D;i,j}$ in $L_2(P^{(j)}_{\theta})$, we obtain correspondingly defined
$p^{(j)}_{1|2:k}$ for $j=1,\ldots k$. Now  $p^{(j)}_{1|2:k}$ being Lebesgue densities of $P^{(0)}_{\theta;1|2:k}(dx_1|y_{2:k})$,
there is a $P_{\theta;2:K}$-null set---for simplicity again $\tilde N_1$---such that for $y \in \tilde N_1^c$,
for each pair $j_1\not=j_2$,
\begin{equation}
p^{(j_1)}_{1|2:k}((x,y)_1)=p^{(j_2)}_{1|2:k}((x,y)_1)\quad [\lambda(dx)]\qquad \mbox{ on }K_{j_1}^c\cap K_{j_2}^c,
\end{equation}
so we may indeed speak of a maximally extended $p_{1|2:k}$ defined on $K^c$.\\[0.3ex]
%
%
%{\bf Skipped steps}\\[0.3ex]
%
%We note that contrary to the proof
%in \citet{Hu:81}, we did not need to show pointwise boundedness of
%$p_1(x_1,y_{2:k})$ by $P_{\theta;\,1|2:k}((-\infty;x_1]|y_{2:k})\int
%g_{a;1}^2\,dP_{\theta}$, nor did we need that $p_1$ tends to 0 for $x_1\to \pm \infty$;
%we thus may skip the approximation to show that  $T_{a;1}1=\int g\,dP_{\theta}=0$ [(4.5) in \citet{Hu:81}].\\[1ex]
%
%
% we see that, on $K^c$,
%  up to a $\lambda \otimes P_{\theta;2:k}$--null set of values $y$, we have shown that
%the function $x_1\mapsto[D_{a;1}p_{1|2:k}]((x_1\!:\!y)_1)$
%is absolutely continuous.
%
%
%%%%%%%%%%%%%%%%%%%%%%%%%%%%%%%%%5
%
{\fbox{\protect\boldmath ${i-1\to i}$\unboldmath}}
%%%%%%%%%%%%%%%%%%%%%%%%%%%%%%%%%
%
%
Assume we have already shown that there is a $P_{\theta;i:k}$-null set $\tilde N_{i-1}$ such that for $y_{i:k}\in \tilde N_{i-1}^c$,
 $P^{(0)}_{\theta}$ admits some conditional density,
\begin{equation}\label{IV}
p_{1:i-1|i:k}(x_{1:i-1},y_{i:k}) \lambda^{i-1}(dx_{1:i-1})
=P^{(0)}_{\theta;\,1:i-1|i:k}(dx_{1:i-1}|y_{i:k}).
\end{equation}
%Assume in particular that
%we have already shown that for $j\leq i-1$,
%for $P_{\theta;\,-j}$--a.e.\ $y\in\R^{k-1}$ the function
%$x\mapsto[D_{a;j }p_{1:j}]((x\!:\!y)_{j})$ is a.c.
%
Arguing just as for $i=1$, we get
\begin{eqnarray*}
T_{a;i}(\varphi)\!\!&\!\!=\!\!&\!\!\int \,D_{a;i}\, \pxI{i}\varphi\,dP_{\theta}=
\int \Jc_{\{D_{a;i}\not=0\}} \,D_{a;i}\, \pxI{i}\varphi\,dP_{\theta}= %\nonumber\\
%\!\!&\!\!=\!\!&\!\!
-\int \varphi  g_{a;i}\,dP_{\theta}.
\end{eqnarray*}
 Thus  using the induction assumptions we proceed as before,
 i.e.; define $\tilde h_i(x_i,y;A)$ for some compact $A$, a the section-wise defined function
 $h_{i}(y_i):= \tilde h_i(x_i,y;A)$, the function $p^{(a;i;A)}_{1:i|i+1:k}$
 which extends to $\R$ giving $p^{(a;i)}_{1:i|i+1:k}$, and where the
 dependence on $a$ may be dropped, giving $p^{(i)}_{1:i|i+1:k}$.
 As this defines the same functional on the set ${\cal D}_{D;i,i}$
 as the Markov kernel $P_{1:i|i+1:k}$, by Lemma~\ref{densel}, $p^{(i)}_{1:i|i+1:k}$
 is a conditional density defined on $K_i^c$. Using  the coordinate permutation argument
 to drop the dependence on $K_i$, we obtain $p_{1:i|i+1:k}$
 defined on $K^c$.

 Hence the induction is complete, and we have shown that
 $P_{\theta}^{(0)}$ admits a $\lambda^k$ density $p_{\theta}$ which we denote by
$$
p_{\theta}(x):=p_{\theta}(x_{1:k}):=p_{\theta;1:k}(x_{1:k}):=p_{1:k}(x_{1:k}).
$$
%\\[0.3ex]
%
%
\noindent{\bf Showing {\mbox{\protect\boldmath $g_{a;i}=0\;[\bar P^{(0)}_{\theta}]$ \unboldmath}:} }%\\[0.3ex]
%\\[0.3ex]
%
Writing \eqref{p1def} and its analogue for general $i$ for any fixed $a\in\R^p$ with $P^{(0)}_{\theta}$ and $\bar P^{(0)}_{\theta}$, we see that
by Fubini, for $y$ outside a $P_{\theta;-i}$-null set,
\begin{equation}\label{a.ci}
[D_{a;i}p_{1:i|i+1:k}]((x\!:\!y)_i):=
\int_{-\infty}^{x} g_{a;i}((z\!:\!y)_i)\,p_{\theta}((z\!:\!y)_i)\lambda(dz)+
\gamma_{a;i}((x\!:\!y)_i)
\end{equation}
with
\begin{equation}\label{gamidef}
\gamma_{a;i}((x\!:\!y)_i):=\int_{-\infty}^{x}  g_{a;i}((z\!:\!y)_i)\,\bar P^{(0)}_{\theta;i|-i}(dz|y)
\end{equation}
We next show that for fixed $a\in \R^p$ and fixed $y$ outside  a $P_{\theta;-i}$-null set,
the value of $\gamma_{a;i}\equiv 0 $:\\
To this end we show that for any Borel subset $B$ of $K$ or equivalently for any proper or improper
interval $I=[l,r]\subset K$,
$$
\int_{I}  g_{a;i}((z\!:\!y)_i)\,\bar P^{(0)}_{\theta;i|-i}(dz|y)=0
$$
Of course, $\int_I \,dP^{(0)}_{\theta;i|-i}(dz|y)=0$, $[P_{\theta;-i}]$.
Consider $\phi_n\in{\cal D}_k$ with
$0\leq \phi_n\leq 1$, $\phi_n\equiv  1$ on $I$ and  $\phi_n\equiv 0$ for
$\{x\in \R^k \,|\, {\rm dist}(x,I)>1/n\}$, and $|\pxI{i}\phi_n| \leq 6n$.
The last bound is chosen according to the bound $|\dot \varphi|\leq 2 c_0 \delta$ from Proposition~\ref{propdense}.  %\\
Then $\phi_n \to \Jc_{I}$ pointwise, hence by dominated convergence and Cauchy-Schwartz we get
$$
\int_I  \phi^2_n \, dP_{\theta;i|-i}^{(0)}=\Lo(n^0),
\qquad\Big|\int  g_{a;i} \phi_n \, dP_{\theta;i|-i}^{(0)}\Big|=\Lo(n^0).
$$
On the other hand let $C:=\max\{|\pxI{i}D_{a;i}|\,\big|\,x \in{\rm supp}(\phi_1)\}$.
Then as $D_{a;i}=0$ on $I$, and because ${\rm supp}(\pxI{i} \phi_n)\subset \{x\in \R^k \,|\, 0<{\rm dist}(x,I)\leq 1/n\}$,
 we have for $x\in {\rm supp}(\phi_1)$ that
$|D_{a;i}(x)|\leq C |x|\leq C/n$, and hence
 $$| \pxI{i} \phi_n D_{a;i}|\leq 6C \Jc_{\{{\rm \Ss supp}(\phi_n)\cap I^c\}}.$$
Thus, due to the shrinking of $\{{\rm supp}(\phi_n)\cap I^c\}$,
for $x\not \in I$, $[\pxI{i} \phi_n D_{a;i}](x) \to 0$.
Furthermore $[\pxI{i}\phi_n D_{a;i}](x) = 0$ on $I$, as $ D_{a;i}(x)=0$ on $I\subset K$ by definition,
hence also $[\pxI{i} \phi_n D_{a;i}]\to 0$ pointwise and with
dominated convergence
$$
\int D_{a;i}\,\pxI{i}\phi_n \, dP_{\theta;i|-i}=\Lo(n^0).
$$
So we have %for $n$ so large that $x_{j}$ is the only zero in  $I_j=[x_{j}-1\!/\!n; x_{j}+1\!/\!n]$
\begin{eqnarray*}
\Lo(n^0)+\int_I  g_{a;i}(z)\,\bar P^{(0)}_{i|-i}(dz)
&=&
\int_{I}  g_{a;i}  \phi_n \,  d\bar P^{(0)}_{i|-i} =
%\int_{I}  g_{a;i}   \phi_n\, (dP_{i|-i}-dP^{(0)}_{i|-i})=\\
%&=&
-\int_{I} \pxI{i} \phi_n D_{a;i}\,dP_{i|-i}-
 \int_{I}  g_{a;i}   \phi_n\, dP^{(0)}_{i|-i} =\Lo(n^0),
\end{eqnarray*}
which implies $\gamma_{a;i}\equiv 0$ and hence, integrating by $\bar P^{(0)}_{\theta;-i}$ over any $A\in\B^{k-1}$, 
%\begin{equation} \label{g=0}
$g_{a;i}=0\;[\bar P^{(0)}_{\theta}]$.
%\end{equation}
Similarly, we obtain
\begin{equation} \label{gij=0}
g_{a;i,j}=0\qquad[\bar P^{(0)}_{\theta}].
\end{equation}
This also entails that
\begin{eqnarray}\label{gip1def}
\int \pxi\!\varphi \,D_{a;j} dP_{\theta}&=&
%\int \pxi\!\varphi \,D_{a;j} dP^{(0)}_{\theta}=
\int \pxi\!\varphi \,D_{a;j} p_{\theta}\,d\lambda^k=%\nonumber \\
%&=&
-\int  \varphi g_{a;i,j}  dP_{\theta} \stackrel{\ssr{\eqref{gij=0}}}{=}
-\int  \varphi g_{a;i,j} p_{\theta}\,d\lambda^k.
\end{eqnarray}
\\[0.3ex]
\noindent{\bf Application of Proposition~\ref{mazja}: }% and \ref{mazja2}}
%\\[0.3ex]
%
From \eqref{gip1def}, we get 
$
%\begin{equation} \label{gipdef}
\int \pxi\!\varphi \,D_{a;j} p_{\theta}\,d\lambda^k=
-\int  \varphi g_{a;i,j}  p_{\theta}\,d\lambda^k
%\end{equation}
$ for all $\varphi \in {\cal D}_k$.
By Definition~\ref{weakderdef}, $g_{a;i,j}  p_{\theta}$ thus is the weak derivative
of $D_{a;j} p_{\theta}$ w.r.t.\ $x_i$. By Proposition~\ref{mazja} there is a modification of
$D_{a;j} p_{\theta}$ on a $\lambda^k$-null set such that this modification---for simplicity again
denoted by $D_{a;j} p_{\theta}$---is a.c.\ in $k$ dimensions.
Hence, for $\lambda^k$ a.e.\ $x$, $D_{a;j} p_{\theta}$ is differentiable w.r.t.\ $x_i$ in the classical sense
 with a derivative coinciding with $g_{a;i,j}  p_{\theta}$ up to a $\lambda^k$-null set.\\
 As $D_{a;j}$ is continuously differentiable,  $p_{\theta}$ is differentiable on $K_j^c$ for $\lambda^k$ a.e.\ $x$,
  and using again all the different $D_{a;j}$, $j=1,\ldots,k$, the same is even true on $K^c$.
\\[0.3ex]
\noindent{\bf Proof of (ii)(a)--(d):} %\\[0.3ex]
 Defining for $\theta \in \Theta$
 \begin{equation}
 f^{(\theta)}:=(p_{\theta}/\dttx)\!\circ\! \tau_{\theta},
 \end{equation}
 and recalling that $P_{\theta}=F\!\circ\!\tath$,
 we see that by the Lebesgue transformation formula, $f^{(\theta)}$ must be a density of $F$, hence the
 index $\theta$ may be dropped, and (ii)(a) follows. 
 Once again by the  transformation formula,
\begin{equation} \label{dichdef}
p_{\theta}=\dttx (f \!\circ\!\iota_{\theta})=\dttx f_{\theta}.
\end{equation}
and thus (ii)(c) holds.
%
%%%%%%%%%%%%%%%%%%%%%%%%%%%%%%%%%%%%%%%%%%%%%%%%%%%%%%%%%%%%%%%%%%%%%%%%%%%%%%%%%%%%%%%%%%%%%%%%%%%%%%%%%%%%%%%%%%%%%%%%%%%
%
For (ii)(b) we consider $\kappa$ defined analogously as for $k=1$ as inverse to $\ell$ from \eqref{ellk}:
We lift \eqref{Idefkk} to $[0,1]^k$, giving
$$\Big(\int_{[0,1]^k} \psi'\kappa q_{\theta} \, d\lambda \Big)^2\leq {\cal I}_{\theta}(F)
                                 \int_{[0,1]^k} \psi^2 \, d[\ell(P_{\theta})] \qquad
\forall\psi \in {\cal C}^{\infty}([0,1]^k,\R),$$
where $q_{\theta}=p_{\theta}\!\circ\!\kappa$, and we have to show that $[\kappa q_{\theta}](u)=0$ for $u\in\partial([0,1]^k)$.
We only show $u_1=1$, all other cases follow similarly.
Let $\psi_n \in{\cal D}_k$, $\psi_n \to \Jc_{\{1\}\times[0,1]^{k-1}}$ in $L_2(\ell(P_{\theta}))$ and pointwise.
Then by Fubini and by integration by parts
\begin{eqnarray*}
&&\int_{[0,1]^{k-1}}\int_0^1 [\partial_{x_1}(\psi_n)\kappa q_{\theta}]((x\!:\!y)_1) \, \lambda(dx)\,\lambda^{k-1}(dy) =\\
&=&\int_{[0,1]^{k-1}} \Big[ \psi_n\kappa q_{\theta} \big|_0^1- \int_0^1 g\!\circ\!\kappa \psi_n \, [\ell(P_{\theta})]_{1|2:k}(dx|y) \Big]\,[\ell(P_{\theta})]_{2:k}(dy)
\end{eqnarray*}
But $\int_{[0,1]^{k}} \psi_n^2 \, d[\ell(P_{\theta})] \to 0$ entails by Fubini,
$\int_0^1 \psi_n^2 \, d[\ell(P_{\theta})]_{1|2:k} \to 0$ $[\ell(P_{\theta})]_{2:k}(dy)$ a.e.\ and by
Cauchy-Schwartz that also $\int_0^1 g\!\circ\!\kappa \psi_n \, d[\ell(P_{\theta})]_{1|2:k}\to 0$
and hence
$$\big([\psi_n\kappa q_{\theta}]((1\!:\!y)_1)+\Lo(n^0)\big)^2\leq\Lo(n^0)\qquad \big[[\ell(P_{\theta})]_{2:k}(dy)\big]$$
and due to continuity of $[\psi_n\kappa q_{\theta}]$, (ii)(b) follows.
%
%%%%%%%%%%%%%%%%%%%%%%%%%%%%%%%%%%%%%%%%%%%%%%%%%%%%%%%%%%%%%%%%%%%%%%%%%%%%%%%%%%%%%%%%%%%%%%%%%%%%%%%%%%%%%%%%%%%%%%%%%%%
%
For (ii)(d) we proceed as in part (ii) $\Rightarrow$ (i)
\begin{eqnarray}
\lefteqn{(\sum_i g_{a;i}-V_a) p_{\theta}=(\sum_i\pxI{i}[D_{a;i}p_{\theta}])-V_a p_{\theta}=%}&&\nonumber\\
%&
\stackrel{\Ss{ \eqref{dichdef}}}{=}
%&
\sum_i\pxI{i}[D_{a;i}\dttx f_{\theta}] -V_a p_{\theta}=}&&\nonumber\\
&=& p_{\theta}(\frac{\sum_i\pxI{i}[D_{a;i}\,\dttx]}{\dttx}+\frac{\sum_i D_{a;i} \pxi f_{\theta}}{f_{\theta}}- V_a )=%\nonumber\\
%&=&
p_{\theta} (\frac{a^{\tau}\pthj \dttx}{\dttx}+\frac{\sum_i D_{a;i} \pxi f_{\theta}}{f_{\theta}})\label{erstes}\\
&=& p_{\theta} (\frac{a^{\tau}\pthj \dttx}{\dttx}+\frac{a^{\tau}\pthj f_{\theta}}{f_{\theta}})=p_{\theta} \frac{a^{\tau}\pthj p_{\theta}}{p_{\theta}}=
p_{\theta} a^{\tau}\Lambda_{\theta}.\label{zweites}
\end{eqnarray}
Now (ii)(c) follows from \eqref{erstes} and the fact that $V_a$  and all $g_{a;i}$ are in $L_2(P_{\theta})$, and
assertions \eqref{ass1} and  \eqref{ass2} from \eqref{zweites}.
\qed \medskip%

%\end{f3proof}
The next corollary shows that  $K$ is uninformative for our problem in the sense that $\bar P^{(0)}_{\theta}$-a.e.\
$\Lambda_{\theta}=0$.
\begin{Cor}\label{clt=0}
Under the assumptions of Theorem~\ref{dasTHM}, setting
\begin{equation} \label{Lbdef}
\Lambda_{\theta}:=-V+\sum_i w_i
\end{equation} with $w_i$ from \eqref{wdef} (for $i=1$) and respectively defined otherwise, it holds that
\begin{equation}\label{lt=0}
\Lambda_{\theta}=0\qquad [\bar P^{(0)}_{\theta}]
\end{equation}
\end{Cor}
\begin{proof}{}
\eqref{Lbdef} is defined according to \eqref{zweites} on $K^c$, and
as \eqref{gij=0} entails $\sum_iw_i =0\;[\bar P^{(0)}_{\theta}]$, the assertion is a direct consequence of
$ x\in K \;\; \iff\;\; D(x)=0 \;\; \iff\;\; \partial_{\theta}\iota_{\theta}(x)=0 \;\; \stackrel{\ssr{\eqref{Vnummer}}}{\Rightarrow} \;\;
V(x)=0. $
%Therefore, we get for any $a\in\R^p$ and $\varphi \in {\cal D}_k$
%\begin{eqnarray*}
%\int \varphi a^{\tau}\Lambda_{\theta} \, dP_{\theta} &=&
%-\int \nabla \varphi^{\tau} D_a^{(k)}+ \varphi V_a \, dP_{\theta}=
%\int  \varphi (\sum_i g_{a;i}-V_a) \, dP_{\theta}=\\
%&=&\int  \varphi (\sum_i g_{a;i}-V_a) \, dP^{(0)}_{\theta}=\int \varphi a^{\tau}\Lambda_{\theta} \, dP^{(0)}_{\theta}
%\end{eqnarray*}
%and as ${\cal D}_k$ is dense in $L_2(P_{\theta})$, \eqref{lt=0} follows.
\end{proof}
\subsection{Proofs of Sections~6}
For the proof of Proposition~\ref{klsl2} we need two lemmas:
\begin{Lem}\label{GHLem}
The multivariate location model~\ref{multilok} is $L_2$-differentiable iff it
is ``partially'' in each coordinate separately, i.e.;
\begin{equation}
\int \Big(\sqrt{f}(x_1,\ldots,x_j+h,\ldots,x_k)-\sqrt{f}(x)(1-\Tfrac{1}{2}\Lambda_{f,j}(x))\Big)^2
\lambda^k(dx)=\Lo(h)
\end{equation}
\end{Lem}
\begin{proof}{to Lemma~\ref{GHLem}}
\citet[Lemma~2.1]{G:H:95}
\end{proof}
\begin{Lem}\label{GHLemScal}
The multivariate location model~\ref{multscal} is $L_2$-differentiable iff it
is ``partially'' in each coordinate separately, i.e.; for each  $i,j=1,\ldots,k$ and each
$A=A^\tau\in \R^{k\times k}$
\begin{equation}
\int\!\! \Big(\!\sqrt{\det}(\EM_k+h\delta_{i,j}A)\sqrt{f}((\EM_k+h\delta_{i,j}A)x)-\sqrt{f}(x)(1+\Tfrac{1}{2}
\Lambda_{\sEM_k}(x))\Big)^2 \lambda^k(dx)\!=\!\Lo(h^2)
\end{equation}
 where  $\delta_{i,j}$ is the matrix in $\R^{k\times k}$ with but $0$ entries except
at position $i,j$.
\end{Lem}

\begin{proof}{to Lemma~\ref{GHLemScal}}
With obvious translation we may parallel \citet[Lemma~2.1]{G:H:95}.
A proof is given in \citet[Lemma~B.3.3]{Ru:01D}.
\end{proof}

\begin{proof}{to Proposition~\ref{klsl2}}
Putting together Lemmas~\ref{GHLem} and \ref{GHLemScal}, we have
reduced the problem to the respective questions in the
one dimensional location resp.\ scale model, which is proven
in \citet{Ha:1972} (one-dimensional location)
and \citet[Ch.2, Sec.3]{Sw:80} (one-dimensional scale);
\citet[Prop.~3.1]{Ru:Ri:10} in addition shows that in the pure scale case, we
may allow for mass in $0$.
\end{proof}
\subsection{Proofs of Section~7}
\begin{proof}{to Proposition~\ref{maxunique}}
For fixed $0\not=a\in\R^p$, the proof goes through word by word as in \citet{Hu:81}, simply replacing
$f_t'$ by $a^{\tau}\partial_{\theta} \tilde f_t$ and $f_t$ by $\tilde f_t$:
By a monotone convergence argument it is shown that we may differentiate twice under the integral sign, giving
$$
\frac{d^2}{dt^2} {\cal I}_{\theta}(F_t;a)=\int 2 \left( \frac{a^{\tau}\partial_{\theta} \tilde f_1}{\tilde f_1}-
\frac{a^{\tau}\partial_{\theta} \tilde f_0}{\tilde f_0}\right)^2 \frac{\tilde f_0^2\tilde f_1^2}{\tilde f_t^3} \, d\lambda^k.
$$
So we conclude that
$a^{\tau}\partial_{\theta}\log \tilde f_0=a^{\tau}\partial_{\theta}\log \tilde f_1$ $\lambda^k(dx)$ a.e., i.e.,
$$
a^{\tau}\partial_{\theta} \iota_{\theta}\frac{\nabla f_0}{f_0} \!\circ\!\iota_{\theta}(x) +a^{\tau}\partial_{\theta} \log |\det \partial_x \iota_{\theta}(x)|=
a^{\tau}\partial_{\theta} \iota_{\theta}\frac{\nabla f_1}{f_1} \!\circ\!\iota_{\theta}(x) +a^{\tau}\partial_{\theta} \log |\det \partial_x \iota_{\theta}(x)|,
$$
where due to (d) up to a $\lambda^k$-null set $\frac{\nabla f_0}{f_0}\!\circ\!\iota_{\theta}=\frac{\nabla  f_1}{f_1}\!\circ\!\iota_{\theta}$,
and hence  up to a $\lambda^k$-null set
$\nabla \log \tilde f_0=\nabla \log \tilde f_1$.
Integrating this out w.r.t\ $x_i$, we get by (c) that $\tilde f_0(x)=c_i(x_{-i}) \tilde f_1(x)$ for $\lambda^{(k-1)}$ almost all $x_{-i }$.
Varying $i$, we see that for some $c>0$, $c_{i}(x_{-i})=c$ for all $i=1,\ldots,k$ and for $\lambda^k$ almost all $x$, and hence
$$
{\cal I}_{\theta}(F_1;a)=\int (\frac{a^{\tau}\partial_{\theta} \tilde f_1}{\tilde f_1})^2 \tilde f_1 \,d\lambda^k=
\int (\frac{a^{\tau}\partial_{\theta} \tilde f_0}{\tilde f_0})^2 c \tilde f_0 \,d\lambda^k=c{\cal I}_{\theta}(F_0;a)
$$
and $c=1$. As this holds for any $0\not=a\in\R^p$, the assertion for $\bar {\cal I}_{\theta}(F)$ follows.
\end{proof}
\begin{proof}{to Proposition~\ref{maxexist}}
As by Proposition~\ref{vlc}, for any $a\in\R^p$ the mapping $F \mapsto {\cal I}_{\theta}(F;a)$ is weakly lower-semicontinuous,
the same goes for the following, recursively defined mappings: Let $a_1\in \R^p$, $|a_1|=1$ realize
$$
\bar {\cal I}_{\theta;1}(F):=\bar {\cal I}_{\theta}(F)=\max{\cal I}_{\theta}(F;a),\quad a\in \R^p, |a|=1
$$
and for $i=2,\ldots, k$, assuming $a_j$ already defined for $j=1,\ldots,i-1$, let $a_i\in \R^p$, $|a_i|=1$ realize
$$
\bar {\cal I}_{\theta;i}(F):=\max{\cal I}_{\theta}(F;a),\quad a\in \sR^p,\, |a|=1,\, a \perp \{a_j\}_{j<i}.
$$
Then each of the $\bar {\cal I}_{\theta;i}(F)$, $i=1,\ldots,k$ is weakly lower-semicontinuous by the same argument as $\bar {\cal I}_{\theta}(F)$
and is strictly positive by assumption.
Hence for each $i=1,\ldots,k$, the mapping $F\mapsto 1\!/\bar {\cal I}_{\theta;i}(F)$ is weakly upper-semicontinuous, and so is the sum
$\sum_i 1\!/\bar {\cal I}_{\theta;i}(F)$. But this sum is just the trace of $[{\cal I}_{\theta}(F)]^{-1}$. The corresponding statement
as to the attainment of the maximum is shown just as Corollary~\ref{minatt}
\end{proof}
\section*{Acknowledgement}
The author wants to thank Helmut Rieder for his helpful suggestions.
\section*{References}
\bibliographystyle{p2ste}
%\bibliographystyle{plainnat}
%\bibliography{litdb1} % literature
%\input{fip}     % plots
%%%%%%%%%%%%%%%%%%%%%%%%%%
% -----------------------------------------------------------------------
%

%Contents:
%
%well definedness of linear operator for appl of Riesz theorem
%definition of FI of scale (1dim)
%definition of FI of location/scale (multi-dim)
%convexity and vague continuity
%minimization of FI by means of Lagrange-Techniques
%achieving the upper bound (tr/maxev)
%connection to differentiability notions in classical analysis (weak(Sobolev), L2-, distributional)
%absolute continuity (along axe
%
% -----------------------------------------------------------------------
\end{document}